\documentclass[12pt,leqno,fleqn]{article}

\usepackage{e-jc,amsmath,amsthm,amssymb,bm,mathrsfs,hyperref,enumerate}

\newtheoremstyle{conv}{}{}{\upshape}{}{\itshape}{}{ }{}
\newtheoremstyle{note}{}{}{\itshape}{}{\itshape}{}{ }{}

\makeatletter
\def\thmhead#1#2#3{%
	\thmnumber{\textup{\mdseries(#2)}}%
	\thmname{\@ifnotempty{#2}{~}#1}%
	\thmnote{{\the\thm@notefont(#3)}}}
\makeatother

\newtheorem{thm}[equation]{Theorem}
\newtheorem{prop}[equation]{Proposition}
\theoremstyle{note}
\newtheorem{note}[equation]{}
\theoremstyle{conv}
\newtheorem{conv}[equation]{}
\theoremstyle{definition}

\newtheorem{rem}[equation]{Remark}
\newtheorem{exam}[equation]{Example}
\newtheorem*{ack}{Acknowledgments}

\newcommand{\gauss}[2]{\genfrac{[}{]}{0pt}{}{#1}{#2}}
\numberwithin{equation}{section}

\title{Vertex subsets with minimal width and dual width in $Q$-polynomial distance-regular graphs}
\author{Hajime Tanaka\thanks{Regular address: Graduate School of Information Sciences, Tohoku University, 6-3-09 Aramaki-Aza-Aoba, Aoba-ku, Sendai 980-8579, Japan} \\
\small Department of Mathematics, University of Wisconsin \\[-0.8ex]
\small 480 Lincoln Drive, Madison, WI 53706, U.S.A. \\
\small \texttt{htanaka@math.is.tohoku.ac.jp}}

\date{\small November 9, 2010 } 

\begin{document}

\maketitle

\begin{abstract}
We study $Q$-polynomial distance-regular graphs from the point of view of what we call \emph{descendents}, that is to say, those vertex subsets with the property that the \emph{width} $w$ and \emph{dual width} $w^*$ satisfy $w+w^*=d$, where $d$ is the diameter of the graph.
We show among other results that a nontrivial descendent with $w\geqslant 2$ is convex precisely when the graph has classical parameters.
The classification of descendents has been done for the $5$ classical families of graphs associated with short regular semilattices.
We revisit and characterize these families in terms of posets consisting of descendents, and extend the classification to all of the $15$ known infinite families with classical parameters and with unbounded diameter.
\end{abstract}

\section{Introduction}

$Q$-\emph{polynomial distance-regular graphs} are thought of as finite/combinatorial analogues of compact symmetric spaces of rank one, and are receiving considerable attention; see e.g., \cite{BI1984B,BCN1989B,Godsil1993B,MT2009EJC} and the references therein.
In this paper, we study these graphs further from the point of view of what we shall call \emph{descendents}, that is to say, those (vertex) subsets with the property that the \emph{width} $w$ and \emph{dual width} $w^*$ satisfy $w+w^*=d$, where $d$ is the diameter of the graph.
See \S \ref{sec: Q-polynomial distance-regular graphs} for formal definitions.
A typical example is a $w$-cube $H(w,2)$ in the $d$-cube $H(d,2)$ $(w\leqslant d)$.

The width and dual width of subsets were introduced and discussed in detail by Brouwer, Godsil, Koolen and Martin \cite{BGKM2003JCTA}, and descendents arise as a special, but very important, case of the theory \cite[\S 5]{BGKM2003JCTA}.
They showed among other results that every descendent is completely regular, and that the induced subgraph is a $Q$-polynomial distance-regular graph if it is connected \cite[Theorems 1--3]{BGKM2003JCTA}.
When the graph is defined on the top fiber of a \emph{short regular semilattice} \cite{Delsarte1976JCTA} (as is the case for the $d$-cube), each object of the semilattice naturally gives rise to a descendent \cite[Theorem 5]{BGKM2003JCTA}.
Hence we may also view descendents as reflecting intrinsic geometric structures of $Q$-polynomial distance-regular graphs.
Incidentally, descendents have been applied to the Erd\H{o}s--Ko--Rado theorem in extremal set theory \cite[Theorem 3]{Tanaka2006JCTA}, and implicitly to the Assmus--Mattson theorem in coding theory \cite[Examples 5.4, 5.5]{Tanaka2009EJC}.

Associated with each $Q$-polynomial distance-regular graph $\Gamma$ is a \emph{Leonard system} \cite{Terwilliger2001LAA,Terwilliger2004LAA,Terwilliger2005DCC}, a linear algebraic framework for a famous theorem of Leonard \cite{Leonard1982SIAM}, \cite[\S 3.5]{BI1984B} which characterizes the terminating branch of the Askey scheme \cite{KLS2010B} of (basic) hypergeometric orthogonal polynomials\footnote{We also allow the specialization $q\rightarrow -1$.} by the duality properties of $\Gamma$.
The starting point of the research presented in this paper is a result of Hosoya and Suzuki \cite[Proposition 1.3]{HS2007EJC} which gives a system of linear equations satisfied by the eigenmatrix of the induced subgraph $\Gamma_Y$ of a descendent $Y$ of $\Gamma$ (when it is connected), and we reformulate this result as the existence of a \emph{balanced} bilinear form between the underlying vector spaces of the Leonard systems associated with $\Gamma$ and $\Gamma_Y$; see \S \ref{sec: basic results}.
Balanced bilinear forms were independently studied in detail in an earlier paper \cite{Tanaka2009LAAb}, and we may derive all the parametric information on descendents from the results of \cite{Tanaka2009LAAb}.

The contents of the paper are as follows.
\S\S \ref{sec: Q-polynomial distance-regular graphs}, \ref{sec: Leonard systems} review basic notation, terminology and facts concerning $Q$-polynomial distance-regular graphs and Leonard systems.
The concept of a descendent is introduced in \S \ref{sec: Q-polynomial distance-regular graphs}.
In \S \ref{sec: basic results}, we relate descendents and balanced bilinear forms.
We give a necessary and sufficient condition on $\Gamma_Y$ to be $Q$-polynomial distance-regular (or equivalently, to be connected) in terms of the parameters of $\Gamma$ (Proposition \eqref{when Gamma_Y is distance-regular}).
In passing, we also show that if $\Gamma_Y$ is connected then a nonempty subset of $Y$ is a descendent of $\Gamma_Y$ precisely when it is a descendent of $\Gamma$ (Proposition \eqref{transitivity}), so that we may define a poset structure on the set of isomorphism classes of $Q$-polynomial distance-regular graphs in terms of isometric embeddings as descendents.
It should be remarked that the parameters of $\Gamma_Y$ in turn determine those of $\Gamma$, provided that the width of $Y$ is at least three; see Proposition \eqref{Phi(Gamma) and Phi(Gamma_Y) determine each other}.

In \S \ref{sec: bipartite case}, we suppose $\Gamma$ is bipartite (with diameter $d$).
The induced subgraph $\Gamma_d^2(x)$ of the distance-$2$ graph of $\Gamma$ on the set $\Gamma_d(x)$ of vertices at distance $d$ from a fixed vertex $x$ is known \cite{Caughman2003EJC} to be distance-regular and $Q$-polynomial.
We show that if $\Gamma_d^2(x)$ has diameter $\lfloor d/2\rfloor$ then for every descendent $Y$ of a halved graph of $\Gamma$, $Y\cap\Gamma_d(x)$ is a descendent of $\Gamma_d^2(x)$ unless it is empty (Proposition \eqref{intersection of descendent of bipartite half with last subconstituent}).
This result will be used in \S \ref{sec: classifications}.

\S \ref{sec: convexity and classical parameters} establishes the main results of the present paper.
Many classical examples of $Q$-polynomial distance-regular graphs have the property that their parameters are expressed in terms of the diameter $d$ and three other parameters $q,\alpha,\beta$ \cite[p.~193]{BCN1989B}.
Such graphs are said to have \emph{classical parameters} $(d,q,\alpha,\beta)$.
There are many results characterizing this property in terms of substructures of graphs; see e.g., \cite[Theorem 7.2]{Weng1998GC}, \cite{Terwilliger199?unpublished}.
We show that a nontrivial descendent $Y$ with width $w\geqslant 2$ is convex (i.e., geodetically closed) precisely when $\Gamma$ has classical parameters $(d,q,\alpha,\beta)$ (Theorem \eqref{convexity and classical parameters}).
Moreover, if this is the case then $\Gamma_Y$ has classical parameters $(w,q,\alpha,\beta)$ (Theorem \eqref{descendents inherit classical parameters}).

In view of this connection with convexity, the remainder of the paper is concerned with graphs with classical parameters.
Currently, there are $15$ known infinite families of such graphs with unbounded diameter, and $5$ of them are associated with short regular semilattices.
The classification of descendents has been done for these $5$ families; by Brouwer et al.~\cite[Theorem 8]{BGKM2003JCTA} for Johnson and Hamming graphs, and by the author \cite[Theorem 1]{Tanaka2006JCTA} for Grassmann, bilinear forms and dual polar graphs.
It turned out that every descendent is isomorphic (under the full automorphism group of the graph) to one afforded by an object of the semilattice.

\S \ref{sec: quantum matroids} is concerned with the $5$ families of ``semilattice-type'' graphs.
We show that if $d\geqslant 4$ then these graphs are characterized by the following properties:
(1) $\Gamma$ has classical parameters; and there is a family $\mathscr{P}$ of descendents of $\Gamma$ such that (2) any two vertices, say, at distance $i$, are contained in a unique descendent in $\mathscr{P}$ with width $i$; and (3) the intersection of two descendents in $\mathscr{P}$ is either empty or a member of $\mathscr{P}$ (Theorem \eqref{characterization of graphs with semilattice-type}).
We remark that if $\mathscr{P}$ is \emph{the} set of descendents of $\Gamma$ then (1), (2) imply (3) (Proposition \eqref{when P is the whole set}).
We shall in fact show that $\mathscr{P}$, together with the partial order defined by reverse inclusion, forms a \emph{regular quantum matroid} \cite{Terwilliger1996ASPM}.
The semilattice structure of $\Gamma$ is then completely recovered from $\mathscr{P}$, and the characterization of $\Gamma$ follows from the classification of nontrivial regular quantum matroids with rank at least four \cite[Theorem 39.6]{Terwilliger1996ASPM}.

\S \ref{sec: classifications} extends the classification of descendents to all of the $15$ families.
We make heavy use of previous work on (noncomplete) convex subgraphs \cite{Lambeck1990D,MPS1993JAC} and maximal cliques \cite{Hemmeter1986EJC,Hemmeter1988EJC,BH1992EJC} in some of these families.
We shall see a strong contrast between the distributions of descendents in the $5$ families of ``semilattice-type'' and the other $10$ families of ``non-semilattice-type''.

The paper ends with an appendix containing necessary data involving the \emph{parameter arrays} (see \S \ref{sec: Leonard systems} for the definition) of Leonard systems.

\section{$Q$-polynomial distance-regular graphs}\label{sec: Q-polynomial distance-regular graphs}

Let $X$ be a finite set and $\mathbb{C}^{X\times X}$ the $\mathbb{C}$-algebra of complex matrices with rows and columns indexed by $X$.
Let $\mathcal{R}=\{R_0,R_1,\dots,R_d\}$ be a set of nonempty symmetric binary relations on $X$.
For each $i$, let $A_i\in\mathbb{C}^{X\times X}$ be the adjacency matrix of the graph $(X,R_i)$.
The pair $(X,\mathcal{R})$ is a (\emph{symmetric}) \emph{association scheme} \emph{with} $d$ \emph{classes} if
\begin{enumerate}[({A}S1)]
\setlength{\itemsep}{0cm}
\item $A_0=I$, the identity matrix;
\item $\sum_{i=0}^dA_i=J$, the all ones matrix;
\item $A_iA_j\in\bm{A}:=\langle A_0,A_1,\dots,A_d\rangle$ for $0\leqslant i,j\leqslant d$.
\end{enumerate}
It follows from (AS1)--(AS3) that $\bm{A}$ is a $(d+1)$-dimensional commutative algebra, called the \emph{Bose--Mesner algebra} of $(X,\mathcal{R})$.
Since $\bm{A}$ is semisimple (as it is closed under conjugate-transposition), there is a basis $\{E_i\}_{i=0}^d$ consisting of the  primitive idempotents of $\bm{A}$, i.e., $E_iE_j=\delta_{ij}E_i$, $\sum_{i=0}^dE_i=I$.
We shall always set $E_0=|X|^{-1}J$.
By (AS2), $\bm{A}$ is also closed under entrywise multiplication, denoted $\circ$.
The $A_i$ are the primitive idempotents of $\bm{A}$ with respect to this multiplication, i.e., $A_i\circ A_j=\delta_{ij}A_i$, $\sum_{i=0}^dA_i=J$.
For convenience, define $A_i=E_i=0$ if $i<0$ or $i>d$.

Let $\mathbb{C}^X$ be the Hermitean space of complex column vectors with coordinates indexed by $X$, so that $\mathbb{C}^{X\times X}$ acts on $\mathbb{C}^X$ from the left.
For each $x\in X$ let $\hat{x}$ be the vector in $\mathbb{C}^X$ with a $1$ in coordinate $x$ and $0$ elsewhere.
The $\hat{x}$ form an orthonormal basis for $\mathbb{C}^X$.

We say $(X,\mathcal{R})$ is $P$-\emph{polynomial with respect to the ordering} $\{A_i\}_{i=0}^d$ if there are integers $a_i,b_i,c_i$ $(0\leqslant i\leqslant d)$ such that $b_d=c_0=0$, $b_{i-1}c_i\ne 0$ $(1\leqslant i\leqslant d)$ and
\begin{equation}\label{three-term recurrence relation}
	A_1A_i=b_{i-1}A_{i-1}+a_iA_i+c_{i+1}A_{i+1} \quad (0\leqslant i\leqslant d)
\end{equation}
where $b_{-1}=c_{d+1}=0$.
Such an ordering is called a $P$-\emph{polynomial ordering}.
It follows that $(X,R_1)$ is regular of valency $k:=b_0$, $a_i+b_i+c_i=k$ $(0\leqslant i\leqslant d)$, $a_0=0$ and $c_1=1$.
Note that $A:=A_1$ generates $\bm{A}$ and hence has $d+1$ distinct eigenvalues $\theta_0:=k,\theta_1,\dots,\theta_d$ so that $A=\sum_{i=0}^d\theta_iE_i$.
Note also that $(X,R_i)$ is the distance-$i$ graph of $(X,R_1)$ for all $i$.
Dually, we say $(X,\mathcal{R})$ is $Q$-\emph{polynomial with respect to the ordering} $\{E_i\}_{i=0}^d$ if there are scalars $a_i^*,b_i^*,c_i^*$ $(0\leqslant i\leqslant d)$ such that $b_d^*=c_0^*=0$, $b_{i-1}^*c_i^*\ne 0$ $(1\leqslant i\leqslant d)$ and
\begin{equation}\label{*three-term recurrence relation}
	E_1\circ E_i=|X|^{-1}(b_{i-1}^*E_{i-1}+a_i^*E_i+c_{i+1}^*E_{i+1}) \quad (0\leqslant i\leqslant d)
\end{equation}
where $b_{-1}^*=c_{d+1}^*=0$.
Such an ordering is called a $Q$-\emph{polynomial ordering}.
It follows that $\mathrm{rank}\, E_1=m:=b_0^*$, $a_i^*+b_i^*+c_i^*=m$ $(0\leqslant i\leqslant d)$, $a_0^*=0$ and $c_1^*=1$.
Note that $|X|E_1$ generates $\bm{A}$ with respect to $\circ$ and hence has $d+1$ distinct entries $\theta_0^*:=m,\theta_1^*,\dots,\theta_d^*$ so that $|X|E_1=\sum_{i=0}^d\theta_i^*A_i$.
We may remark that
\begin{equation}\label{Norton}
	\langle (E_1\mathbb{C}^X)\circ(E_i\mathbb{C}^X)\rangle\subseteq E_{i-1}\mathbb{C}^X+E_i\mathbb{C}^X+E_{i+1}\mathbb{C}^X \quad (0\leqslant i\leqslant d).
\end{equation}
See e.g., \cite[p.~126, Proposition 8.3]{BI1984B}.

A connected simple graph $\Gamma$ with vertex set $V \Gamma = X$, diameter $d$ and path-length distance $\partial$ is called \emph{distance-regular} if the distance-$i$ relations $(0\leqslant i\leqslant d)$ together form an association scheme.
Hence $P$-polynomial association schemes, with specified $P$-polynomial ordering, are in bijection with distance-regular graphs, and we shall say, e.g., that $\Gamma$ \emph{is} $Q$-polynomial, and so on.
The sequence
\begin{equation}
	\iota(\Gamma)=\{b_0,b_1,\dots,b_{d-1};c_1,c_2,\dots,c_d\}
\end{equation}
is called the \emph{intersection array} of $\Gamma$.
Given $x\in X$, we write $\Gamma_i(x)=\{y\in X:\partial(x,y)=i\}$, $k_i=|\Gamma_i(x)|$ $(0\leqslant i\leqslant d)$.
We abbreviate $\Gamma(x)=\Gamma_1(x)$.

We say $\Gamma$ has \emph{classical parameters} $(d,q,\alpha,\beta)$ \cite[p.~193]{BCN1989B} if
\begin{equation}\label{classical parameters}
	b_i=\biggl(\gauss{d}{1}_q-\gauss{i}{1}_q\biggr)\biggl(\beta-\alpha\gauss{i}{1}_q\biggr), \quad c_i=\gauss{i}{1}_q\biggl(1+\alpha\gauss{i-1}{1}_q\biggr) \quad (0\leqslant i\leqslant d)
\end{equation}
where $\gauss{i}{j}_q$ is the $q$-binomial coefficient.
By \cite[Proposition 6.2.1]{BCN1989B}, $q$ is an integer $\ne 0,-1$.
In this case $\Gamma$ has a $Q$-polynomial ordering $\{E_i\}_{i=0}^d$ which we call \emph{standard}, such that
\begin{equation}\label{theta* in terms of classical parameters}
	\theta_i^*=\xi^*\gauss{d-i}{1}_q+\zeta^* \quad (0\leqslant i\leqslant d)
\end{equation}
for some $\xi^*,\zeta^*$ with $\xi^*\ne 0$ \cite[Corollary 8.4.2]{BCN1989B}.

For the rest of this section, suppose further that $\Gamma$ is $Q$-polynomial with respect to the ordering $\{E_i\}_{i=0}^d$.
For the moment fix a ``base vertex'' $x\in X$, and let $E_i^*=E_i^*(x):=\mathrm{diag}(A_i\hat{x})$, $A_i^*=A_i^*(x):=|X|\,\mathrm{diag}(E_i\hat{x})$ $(0\leqslant i\leqslant d)$.\footnote{For a complex matrix $B$, it is customary that $B^*$ denotes the conjugate transpose of $B$. It should be stressed that we are \emph{not} using this convention.}
We abbreviate $A^*=A_1^*$.
Note that $E_i^*E_j^*=\delta_{ij}E_i^*$, $\sum_{i=0}^dE_i^*=I$.
The $E_i^*$ and the $A_i^*$ form two bases for the \emph{dual Bose--Mesner algebra} $\bm{A}^*=\bm{A}^*(x)$ \emph{with respect to} $x$.
Note also that $A^*$ generates $\bm{A}^*$ and $A^*=\sum_{i=0}^d\theta_i^*E_i^*$.
The \emph{Terwilliger} (or \emph{subconstituent}) \emph{algebra} $\bm{T}=\bm{T}(x)$ \emph{of} $\Gamma$ \emph{with respect to} $x$ is the subalgebra of $\mathbb{C}^{X\times X}$ generated by $\bm{A}$, $\bm{A}^*$ \cite{Terwilliger1992JAC,Terwilliger1993JACa,Terwilliger1993JACb}.
Since $\bm{T}$ is closed under conjugate-transposition, it is semisimple and any two nonisomorphic irreducible $\bm{T}$-modules in $\mathbb{C}^X$ are orthogonal.

Let $Y$ be a nonempty subset of $X$ and $\hat{Y}=\sum_{x\in Y}\hat{x}$ its characteristic vector.
We let $\Gamma_Y$ denote the subgraph of $\Gamma$ induced on $Y$.
Set $Y_i=\{x\in X:\partial(x,Y)=i\}$ $(0\leqslant i\leqslant\rho)$, where $\rho=\max\{\partial(x,Y):x\in X\}$ is the covering radius of $Y$.
Note that $\sum_{i=0}^{\rho}\hat{Y}_i=\hat{X}$.
We call $Y$ \emph{completely regular} if $\langle\hat{Y}_0,\hat{Y}_1,\dots,\hat{Y}_{\rho}\rangle$ is an $\bm{A}$-module.
Brouwer et al.~\cite{BGKM2003JCTA} defined the \emph{width} $w$ and \emph{dual width} $w^*$ of $Y$ as follows:
\begin{equation}
	w=\max\{i:\hat{Y}^{\mathsf{T}}A_i\hat{Y}\ne 0\}, \quad w^*=\max\{i:\hat{Y}^{\textsf{T}}E_i\hat{Y}\ne 0\}.
\end{equation}
They showed (among other results) that

\begin{thm}[{\cite[\S 5]{BGKM2003JCTA}}]\label{fundamental inequality}
We have $w+w^*\geqslant d$.
If equality holds then $Y$ is completely regular with covering radius $w^*$, and $(Y,\mathcal{R}^Y)$ forms a $Q$-polynomial association scheme with $w$-classes, where $\mathcal{R}^Y=\{R_i\cap(Y\times Y):0\leqslant i\leqslant w\}$.
\end{thm}

We call $Y$ a \emph{descendent} of $\Gamma$ if $w+w^*=d$.
The descendents with $w=0$ are precisely the singletons, and $X$ is the unique descendent with $w=d$; we shall refer to these cases as \emph{trivial} and say \emph{nontrivial} otherwise.
By \eqref{fundamental inequality} it follows that

\begin{thm}[{\cite[Theorem 3]{BGKM2003JCTA}}]\label{connectivity implies distance-regularity}
Suppose $w+w^*=d$.
If $\Gamma_Y$ is connected then it is a $Q$-polynomial distance-regular graph with diameter $w$.
\end{thm}

We comment on the $Q$-polynomiality of $(Y,\mathcal{R}^Y)$ stated in \eqref{fundamental inequality}.
Suppose $w+w^*=d$ and let $\bm{A}'$ be the Bose--Mesner algebra of $(Y,\mathcal{R}^Y)$.
For every $B\in\bm{A}$, let $\breve{B}$ be the principal submatrix of $B$ corresponding to $Y$.
Brouwer et al.~\cite[\S 4]{BGKM2003JCTA} observed 
\begin{equation}\label{orthogonal if |i-j|>w*}
	\breve{E}_i\breve{E}_j=0 \quad \text{if} \ |i-j|>w^*,
\end{equation}
and then showed that $\langle\breve{E}_0,\breve{E}_1,\dots,\breve{E}_i\rangle$ is an ideal of $\bm{A}'$ for all $i$. Hence we get a $Q$-polynomial ordering $\{E_i'\}_{i=0}^w$ of the primitive idempotents of $\bm{A}'$ such that
\begin{equation}\label{ideals of A'}
	\langle E_0',E_1',\dots,E_i'\rangle=\langle\breve{E}_0,\breve{E}_1,\dots,\breve{E}_i\rangle \quad (0\leqslant i\leqslant w).
\end{equation}

Throughout we shall adopt the following convention and retain the notation of \S \ref{sec: Q-polynomial distance-regular graphs}:

\begin{conv}\label{general convention}
For the rest of this paper, we assume $\Gamma$ is distance-regular with diameter $d\geqslant 3$ and is $Q$-polynomial with respect to the ordering $\{E_i\}_{i=0}^d$.
Unless otherwise stated, $Y$ will denote a nontrivial descendent of $\Gamma$ with width $w$ and dual width $w^*=d-w$.
\end{conv}

\section{Leonard systems}\label{sec: Leonard systems}

Let $d$ be a positive integer.
Let $\mathfrak{A}$ be a $\mathbb{C}$-algebra isomorphic to the full matrix algebra $\mathbb{C}^{(d+1)\times(d+1)}$ and $W$ an irreducible left $\mathfrak{A}$-module.
Note that $W$ is unique up to isomorphism and $\dim W=d+1$.
An element $\mathfrak{a}$ of $\mathfrak{A}$ is called \emph{multiplicity-free} if it has $d+1$ mutually distinct eigenvalues.
Suppose $\mathfrak{a}$ is multiplicity-free and let $\{\theta_i\}_{i=0}^d$ be an ordering of the eigenvalues of $\mathfrak{a}$.
Then there is a sequence of elements $\{\mathfrak{e}_i\}_{i=0}^d$ in $\mathfrak{A}$ such that (i)~$\mathfrak{a}\mathfrak{e}_i=\theta_i\mathfrak{e}_i$; (ii) $\mathfrak{e}_i\mathfrak{e}_j=\delta_{ij}\mathfrak{e}_i$; (iii) $\sum_{i=0}^d\mathfrak{e}_i=\mathfrak{1}$ where $\mathfrak{1}$ is the identity of $\mathfrak{A}$.
We call $\mathfrak{e}_i$ the \emph{primitive idempotent} of $\mathfrak{a}$ associated with $\theta_i$.
Note that $\mathfrak{a}$ generates $\langle\mathfrak{e}_0,\mathfrak{e}_1,\dots,\mathfrak{e}_d\rangle$.

A \emph{Leonard system} in $\mathfrak{A}$ \cite[Definition 1.4]{Terwilliger2001LAA} is a sequence
\begin{equation}\label{Leonard system}
	\Phi=\left(\mathfrak{a};\mathfrak{a}^*;\{\mathfrak{e}_i\}_{i=0}^d;\{\mathfrak{e}_i^*\}_{i=0}^d\right)
\end{equation}
satisfying the following axioms (LS1)--(LS5):
\begin{enumerate}[(LS1)]
\setlength{\itemsep}{0cm}
\item Each of $\mathfrak{a},\mathfrak{a}^*$ is a multiplicity-free element in $\mathfrak{A}$.
\item $\{\mathfrak{e}_i\}_{i=0}^d$ is an ordering of the primitive idempotents of $\mathfrak{a}$.
\item $\{\mathfrak{e}_i^*\}_{i=0}^d$ is an ordering of the primitive idempotents of $\mathfrak{a}^*$.
\item $\mathfrak{e}_i^*\mathfrak{a}\mathfrak{e}_j^*=\begin{cases} 0 & \text{if } |i-j|>1 \\ \ne 0 & \text{if } |i-j|=1 \end{cases} \quad (0\leqslant i,j\leqslant d)$.
\item $\mathfrak{e}_i\mathfrak{a}^*\mathfrak{e}_j=\begin{cases} 0 & \text{if } |i-j|>1 \\ \ne 0 & \text{if } |i-j|=1 \end{cases} \quad (0\leqslant i,j\leqslant d)$.
\end{enumerate}
We call $d$ the \emph{diameter} of $\Phi$.
For convenience, define $\mathfrak{e}_i=\mathfrak{e}_i^*=0$ if $i<0$ or $i>d$.
Observe%
\begin{equation}\label{ith component of 0*}
	\mathfrak{e}_0^*W+\mathfrak{e}_1^*W+\dots+\mathfrak{e}_i^*W=\mathfrak{e}_0^*W+\mathfrak{a}\mathfrak{e}_0^*W+\dots+\mathfrak{a}^i\mathfrak{e}_0^*W \quad (0\leqslant i\leqslant d).
\end{equation}

A Leonard system $\Psi$ in a $\mathbb{C}$-algebra $\mathfrak{B}$ is \emph{isomorphic} to $\Phi$ if there is a $\mathbb{C}$-algebra isomorphism $\sigma:\mathfrak{A}\rightarrow\mathfrak{B}$ such that $\Psi=\Phi^{\sigma}:=\left(\mathfrak{a}^{\sigma};\mathfrak{a}^{*\sigma};\{\mathfrak{e}_i^{\sigma}\}_{i=0}^d;\{\mathfrak{e}_i^{*\sigma}\}_{i=0}^d\right)$.
Let $\xi,\xi^*,\zeta,\zeta^*$ be scalars with $\xi,\xi^*\ne 0$.
Then
\begin{equation}\label{affine transformation}
	\left(\xi\mathfrak{a}+\zeta\mathfrak{1};\xi^*\mathfrak{a}^*+\zeta^*\mathfrak{1};\{\mathfrak{e}_i\}_{i=0}^d;\{\mathfrak{e}_i^*\}_{i=0}^d\right)
\end{equation}
is a Leonard system in $\mathfrak{A}$, called an \emph{affine transformation} of $\Phi$.
We say $\Phi$, $\Psi$ are \emph{affine-isomorphic} if $\Psi$ is isomorphic to an affine transformation of $\Phi$.
The \emph{dual} of $\Phi$ is
\begin{align}
	\Phi^*&=\left(\mathfrak{a}^*;\mathfrak{a};\{\mathfrak{e}_i^*\}_{i=0}^d;\{\mathfrak{e}_i\}_{i=0}^d\right).
\end{align}
For any object $f$ associated with $\Phi$, we shall occasionally denote by $f^*$ the corresponding object for $\Phi^*$; an example is $\mathfrak{e}_i^*(\Phi)=\mathfrak{e}_i(\Phi^*)$.
Note that $(f^*)^*=f$.

\begin{exam}\label{Phi(Gamma)}
With reference to \eqref{general convention}, fix the base vertex $x\in X$ and let $W=\bm{A}\hat{x}=\bm{A}^*\hat{X}$ be the \emph{primary} $\bm{T}$-module \cite[Lemma 3.6]{Terwilliger1992JAC}.
Set $\mathfrak{a}=A|_W$, $\mathfrak{a}^*=A^*|_W$, $\mathfrak{e}_i=E_i|_W$, $\mathfrak{e}_i^*=E_i^*|_W$ $(0\leqslant i\leqslant d)$.
Then $\Phi=\Phi(\Gamma;x):=\left(\mathfrak{a};\mathfrak{a}^*;\{\mathfrak{e}_i\}_{i=0}^d;\{\mathfrak{e}_i^*\}_{i=0}^d\right)$ is a Leonard system.
See \cite[Theorem 4.1]{Terwilliger1993JACa}, \cite{Cerzo2010LAA}.
We remark that $\Phi(\Gamma;x)$ does not depend on $x$ up to isomorphism, so that we shall write $\Phi(\Gamma)=\Phi(\Gamma;x)$ where the context allows.
\end{exam}

\begin{exam}\label{Phi(W)}
More generally, let $W$ be any irreducible $\bm{T}$-module.
We say $W$ is \emph{thin} if $\dim E_i^*W\leqslant 1$ for all $i$.
Suppose $W$ is thin.
Then $W$ is \emph{dual thin}, i.e., $\dim E_iW\leqslant 1$ for all $i$, and there are integers $\epsilon$ (\emph{endpoint}), $\epsilon^*$ (\emph{dual endpoint}), $\delta$ (\emph{diameter}) such that $\{i:E_i^*W\ne 0\}=\{\epsilon,\epsilon+1,\dots,\epsilon+\delta\}$, $\{i:E_iW\ne 0\}=\{\epsilon^*,\epsilon^*+1,\dots,\epsilon^*+\delta\}$ \cite[Lemmas 3.9, 3.12]{Terwilliger1992JAC}.\footnote{In \cite{Terwilliger1992JAC,Terwilliger1993JACa,Terwilliger1993JACb}, $\epsilon$ and $\epsilon^*$ are called the dual endpoint and endpoint of $W$, respectively.}
Set $\mathfrak{a}=A|_W$, $\mathfrak{a}^*=A^*|_W$, $\mathfrak{e}_i=E_{\epsilon^*+i}|_W$, $\mathfrak{e}_i^*=E_{\epsilon+i}^*|_W$ $(0\leqslant i\leqslant \delta)$.
Then $\Phi=\Phi(W):=\left(\mathfrak{a};\mathfrak{a}^*;\{\mathfrak{e}_i\}_{i=0}^{\delta};\{\mathfrak{e}_i^*\}_{i=0}^{\delta}\right)$ is a Leonard system.
Note that $\Phi(\Gamma;x)=\Phi(\bm{A}\hat{x})$.
\end{exam}

For $0\leqslant i\leqslant d$ let $\theta_i$ (resp. $\theta_i^*$) be the eigenvalue of $\mathfrak{a}$ (resp. $\mathfrak{a}^*$) associated with $\mathfrak{e}_i$ (resp. $\mathfrak{e}_i^*$).
By \cite[Theorem 3.2]{Terwilliger2001LAA} there are scalars $\varphi_i$ $(1\leqslant i\leqslant d)$ and a $\mathbb{C}$-algebra isomorphism $\natural:\mathfrak{A}\rightarrow\mathbb{C}^{(d+1)\times(d+1)}$ such that
$\mathfrak{a}^{\natural}$ (resp. $\mathfrak{a}^{*\natural}$) is the lower (resp. upper) bidiagonal matrix with diagonal entries $(\mathfrak{a}^{\natural})_{ii}=\theta_i$ (resp. $(\mathfrak{a}^{*\natural})_{ii}=\theta_i^*$) $(0\leqslant i\leqslant d)$ and subdiagonal (resp. superdiagonal) entries $(\mathfrak{a}^{\natural})_{i,i-1}=1$ (resp. $(\mathfrak{a}^{*\natural})_{i-1,i}=\varphi_i$) $(1\leqslant i\leqslant d)$.
We let $\phi_i=\varphi_i(\Phi^{\Downarrow})$ $(1\leqslant i\leqslant d)$, where $\Phi^{\Downarrow}=\left(\mathfrak{a};\mathfrak{a}^*;\{\mathfrak{e}_{d-i}\}_{i=0}^d;\{\mathfrak{e}_i^*\}_{i=0}^d\right)$.\footnote{\label{D4 action}Viewed as permutations on all Leonard systems, $*$ and $\Downarrow$ generate a dihedral group with $8$ elements which plays a fundamental role in the theory of Leonard systems.}
The \emph{parameter array} of $\Phi$ is
\begin{equation}\label{parameter array}
	p(\Phi)=\left(\{\theta_i\}_{i=0}^d;\{\theta_i^*\}_{i=0}^d;\{\varphi_i\}_{i=1}^d;\{\phi_i\}_{i=1}^d\right).
\end{equation}
By \cite[Theorem 1.9]{Terwilliger2001LAA}, the isomorphism class of $\Phi$ is determined by $p(\Phi)$.
In \cite{Terwilliger2005DCC}, $p(\Phi)$ is given in closed form; see also \eqref{list of parameter arrays}.
Note that the parameter array of \eqref{affine transformation} is given by
\begin{equation}\label{parameter array of affine transformation}
	\left(\{\xi\theta_i+\zeta\}_{i=0}^d;\{\xi^*\theta_i^*+\zeta^*\}_{i=0}^d;\{\xi\xi^*\varphi_i\}_{i=1}^d;\{\xi\xi^*\phi_i\}_{i=1}^d\right).
\end{equation}

Let $u$ be a nonzero vector in $\mathfrak{e}_0W$.
Then $\{\mathfrak{e}_i^*u\}_{i=0}^d$ is a basis for $W$ \cite[Lemma 10.2]{Terwilliger2004LAA}.
Define the scalars $a_i,b_i,c_i$ $(0\leqslant i\leqslant d)$ by $b_d=c_0=0$ and
\begin{equation}
	\mathfrak{a}\mathfrak{e}_i^*u=b_{i-1}\mathfrak{e}_{i-1}^*u+a_i\mathfrak{e}_i^*u+c_{i+1}\mathfrak{e}_{i+1}^*u \quad (0\leqslant i\leqslant d)
\end{equation}
where $b_{-1}=c_{d+1}=0$.
By \cite[Theorem 17.7]{Terwilliger2004LAA} it follows that
\begin{equation}\label{bi, ci}
	b_i=\varphi_{i+1}\frac{\tau_i^*(\theta_i^*)}{\tau_{i+1}^*(\theta_{i+1}^*)}, \quad c_i=\phi_i\frac{\eta_{d-i}^*(\theta_i^*)}{\eta_{d-i+1}^*(\theta_{i-1}^*)} \quad (0\leqslant i\leqslant d)
\end{equation}
where $\theta_{-1}^*,\theta_{d+1}^*$ are indeterminates, $\varphi_{d+1}=\phi_0=0$ and
\begin{equation}
	\tau_i(\lambda)=\prod_{l=0}^{i-1}(\lambda-\theta_l), \quad \eta_i(\lambda)=\prod_{l=0}^{i-1}(\lambda-\theta_{d-l}) \quad (0\leqslant i\leqslant d).
\end{equation}

\begin{exam}\label{parameters of Phi and Gamma}
Let $\Phi=\Phi(\Gamma)$ be as in \eqref{Phi(Gamma)}.
Then $b_i(\Gamma)=b_i(\Phi)$, $b_i^*(\Gamma)=b_i^*(\Phi)$, $c_i(\Gamma)=c_i(\Phi)$, $c_i^*(\Gamma)=c_i^*(\Phi)$ $(0\leqslant i\leqslant d)$.
See \cite[Theorem 4.1]{Terwilliger1993JACa}.
\end{exam}

Let $\Phi'=\bigl(\mathfrak{a}';\mathfrak{a}^{*\prime};\{\mathfrak{e}_i'\}_{i=0}^{d'};\{\mathfrak{e}_i^{*\prime}\}_{i=0}^{d'}\bigr)$ be another Leonard system with diameter $d'\leqslant d$ and $W'=W(\Phi')$ the vector space underlying $\Phi'$.
Given an integer $\rho$ ($0\leqslant\rho\leqslant d-d'$), a nonzero bilinear form $(\cdot |\cdot): W\times W'\rightarrow\mathbb{C}$ is called $\rho$-\emph{balanced with respect to} $\Phi$, $\Phi'$ if
\begin{enumerate}[(B1)]
\setlength{\itemsep}{0cm}
\item $(\mathfrak{e}_i^*W|\mathfrak{e}_j^{*\prime}W')=0$ if $i-\rho\ne j$ \ $(0\leqslant i\leqslant d, \ 0\leqslant j\leqslant d')$;
\item $(\mathfrak{e}_iW|\mathfrak{e}_j'W')=0$ if $i<j$ or $i>j+d-d'$ \ $(0\leqslant i\leqslant d, \ 0\leqslant j\leqslant d')$.
\end{enumerate}
We call $\Phi'$ a $\rho$-\emph{descendent} of $\Phi$ whenever such a form exists.
The $\rho$-\emph{descendents} of $\Phi$ are completely classified; see \eqref{characterization of bilinear form in parametric form}.
In particular, by \eqref{uniqueness of expression}, \eqref{characterization of bilinear form in parametric form}, \eqref{parameter array of affine transformation} it follows that

\begin{prop}\label{uniqueness of descendents}
Let $d,d',\rho$ be integers such that $1\leqslant d'\leqslant d$, $0\leqslant\rho\leqslant d-d'$.
Then a Leonard system with diameter $d$ has at most one $\rho$-descendent with diameter $d'$ up to affine isomorphism.
Conversely, if $d'\geqslant 3$ then a Leonard system with diameter $d'$ is a $\rho$-descendent of at most one Leonard system with diameter $d$ up to affine isomorphism.
\end{prop}

\section{Basic results concerning descendents}\label{sec: basic results}

With reference to \eqref{general convention}, we begin with the following observation (cf.~\cite[p.~73]{HS2007EJC}):

\begin{note}\label{Hosoya--Suzuki observation}
With the notation of \S \ref{sec: Q-polynomial distance-regular graphs}, for any $i,j$ $(0\leqslant i\leqslant d,\ 0\leqslant j\leqslant w)$ we have
\begin{equation*}
	\breve{E}_iE_j'=\begin{cases} 0 & \text{if}\ i<j\ \text{or}\ i>j+w^*, \\ \ne 0 & \text{if}\ i=j\ \text{or}\ i=j+w^*. \end{cases}
\end{equation*}
\end{note}

\begin{proof}
By \eqref{orthogonal if |i-j|>w*}, \eqref{ideals of A'} it follows that $\breve{E}_i\in\langle E_{i-w^*}',\dots,E_i'\rangle$, so that $\breve{E}_iE_j'=0$ if $i<j$ or $i>j+w^*$.
By \eqref{ideals of A'} we also find $\breve{E}_jE_j'\ne 0$.
Note that $E_{j+w^*}\circ E_j\in\langle E_{w^*},\dots,E_d\rangle$ and the coefficient of $E_{w^*}$ in $E_{j+w^*}\circ E_j$ is nonzero.
Hence $\mathrm{trace}(\breve{E}_{j+w^*}\breve{E}_j)=\hat{Y}^{\mathsf{T}}(E_{j+w^*}\circ E_j)\hat{Y}\ne 0$.
It follows that $\breve{E}_{j+w^*}\breve{E}_j\ne 0$ and therefore $\breve{E}_{j+w^*}E_j'\ne 0$ by \eqref{orthogonal if |i-j|>w*}, \eqref{ideals of A'}.
\end{proof}

As mentioned in the introduction, Hosoya and Suzuki translated \eqref{Hosoya--Suzuki observation} into a system of linear equations satisfied by the eigenmatrix of $(Y,\mathcal{R}^Y)$; see \cite[Proposition 1.3]{HS2007EJC}.
We now show how descendents are related to balanced bilinear forms:

\begin{prop}\label{when Gamma_Y is distance-regular}
Pick any $x\in Y$, and let the parameter array of $\Phi=\Phi(\Gamma;x)$ be given as in \eqref{list of parameter arrays}.
Suppose $w>1$.
Then $\Gamma_Y$ is a $Q$-polynomial distance-regular graph precisely for Cases I, IA, II, IIA, IIB, IIC; or Case III with $w^*$ even.
If this is the case then the bilinear form $(\cdot|\cdot):\bm{A}\hat{x}\times\bm{A}'\hat{x}\rightarrow\mathbb{C}$ defined by $(u|u')=u^{\mathsf{T}}u'$ is $0$-balanced with respect to $\Phi$, $\Phi(\Gamma_Y;x)$.
\end{prop}

\begin{proof}
Write $W=\bm{A}\hat{x}$, $W'=\bm{A}'\hat{x}$.
Note that $(E_iW|E_j'W')=0$ whenever $\breve{E}_iE_j'=0$.
Hence it follows from \eqref{Hosoya--Suzuki observation} that $(E_iW|E_j'W')=0$ if $i<j$ or $i>j+w^*$.
Suppose $\Gamma_Y$ is distance-regular.
Then by these comments we find that $(\cdot|\cdot)$ is $0$-balanced with respect to $\Phi$, $\Phi(\Gamma_Y;x)$.
By virtue of \eqref{characterization of bilinear form in parametric form}, $w^*$ must be even if $p(\Phi)$ is of Case III.
Conversely, suppose $p(\Phi)$ (and $w^*$) satisfies one of the cases mentioned in \eqref{when Gamma_Y is distance-regular}.
Then by \cite[Theorem 7.3]{Tanaka2009LAAb} there is a Leonard system $\Phi'=\bigl(\mathfrak{a}';\mathfrak{a}^{*\prime};\{\mathfrak{e}_i'\}_{i=0}^w;\{\mathfrak{e}_i^{*\prime}\}_{i=0}^w\bigr)$ with $W(\Phi')=W'$ such that $\mathfrak{e}_i'=E_i'|_{W'}$, $\mathfrak{e}_i^{*\prime}=E_i^{*\prime}|_{W'}$ where $E_i^{*\prime}=\mathrm{diag}(\breve{A}_i\hat{x})$ $(0\leqslant i\leqslant w)$.
Note that $\breve{A}|_{W'}\in\langle\mathfrak{e}_0',\mathfrak{e}_1',\dots,\mathfrak{e}_w'\rangle$, so that $\breve{A}|_{W'}$ is a polynomial in $\mathfrak{a}'$.
Since $\breve{A}\mathfrak{e}_0^{*\prime}W' = \langle \breve{A}\hat{x} \rangle =\mathfrak{e}_1^{*\prime}W'$, it follows from \eqref{ith component of 0*} that $\breve{A}|_{W'}=\xi\mathfrak{a}'+\zeta\mathfrak{1}'$ for some $\xi,\zeta\in\mathbb{C}$ with $\xi\ne 0$, where $\mathfrak{1}'$ is the identity operator on $W'$.
Hence $\langle E_i^{*\prime}\breve{A}^i\hat{x} \rangle =\mathfrak{e}_i^{*\prime}(\mathfrak{a}')^i\mathfrak{e}_0^{*\prime}W'=\mathfrak{e}_i^{*\prime}W'= \langle \breve{A}_i\hat{x} \rangle$ for all $i$.
In particular, $\Gamma_Y$ is connected and thus distance-regular by \eqref{connectivity implies distance-regularity}.
\end{proof}

By \eqref{when Gamma_Y is distance-regular}, connectivity and therefore distance-regularity of $\Gamma_Y$ can be read off the parameters of $\Gamma$.
We have a comment.

\begin{prop}\label{Phi(Gamma) and Phi(Gamma_Y) determine each other}
Suppose $\Gamma_Y$ is distance-regular.
Then $\Phi(\Gamma_Y)$ is uniquely determined by $\Phi(\Gamma)$ up to isomorphism.
Conversely, if $w\geqslant 3$ then $\Phi(\Gamma_Y)$ uniquely determines $\Phi(\Gamma)$ up to isomorphism.
\end{prop}

\begin{proof}
By \eqref{when Gamma_Y is distance-regular}, $\Phi(\Gamma_Y)$ is a $0$-descendent of $\Phi(\Gamma)$.
Hence the result follows from \eqref{uniqueness of descendents} together with the additional normalizations $b_0(\Psi)=\theta_0(\Psi)$, $b_0^*(\Psi)=\theta_0^*(\Psi)$, $c_1(\Psi)=c_1^*(\Psi)=1$ for each $\Psi\in\{\Phi(\Gamma),\Phi(\Gamma_Y)\}$.
\end{proof}

The following is another consequence of \eqref{Hosoya--Suzuki observation}:

\begin{prop}\label{transitivity}
Suppose $\Gamma_Y$ is distance-regular.
Then a nonempty subset of $Y$ is a descendent of $\Gamma_Y$ if and only if it is a descendent of $\Gamma$.
\end{prop}

\begin{proof}
Let $Z\subseteq Y$ have dual width $w^{*\prime}$ in $\Gamma_Y$.
For $0\leqslant i\leqslant w$, by \eqref{Hosoya--Suzuki observation} we find $\breve{E}_{i+w^*} \in \langle E_i',\dots,E_w' \rangle$ and the coefficient of $E_i'$ in $\breve{E}_{i+w^*}$ is nonzero.
Since $\hat{Z}^{\mathsf{T}}E_i\hat{Z}=\hat{Z}^{\mathsf{T}}\breve{E}_i\hat{Z}$, it follows that $Z$ has dual width $w^{*\prime}+w^*$ in $\Gamma$.
\end{proof}

\begin{rem}\label{from Gamma_Y to Gamma}
Let $\mathscr{L}$ be the set of isomorphism classes of $Q$-polynomial distance-regular graphs with diameter at least three.
For two isomorphism classes $[\Gamma], [\Delta]\in\mathscr{L}$, write $[\Delta]\preccurlyeq[\Gamma]$ if $[\Delta]=[\Gamma_Y]$ for some descendent $Y$ of $\Gamma$.
Then by \eqref{transitivity} it follows that $\preccurlyeq$ is a partial order on $\mathscr{L}$.
Determining all descendents of $\Gamma$ amounts to describing the order ideal generated by $[\Gamma]$.
Conversely, given $[\Gamma]\in\mathscr{L}$, it is a problem of some significance to determine the filter generated by $[\Gamma]$, i.e.,  $V_{[\Gamma]}=\{[\Delta]\in\mathscr{L}:[\Gamma]\preccurlyeq[\Delta]\}$.
\end{rem}

Let $A^*(Y)=|X||Y|^{-1}\mathrm{diag}(E_1\hat{Y})$, $E_i^*(Y)=\mathrm{diag}(\hat{Y}_i)$ $(0\leqslant i\leqslant w^*)$ where $Y_i=\{x\in X:\partial(x,Y)=i\}$.
Let $\tilde{W}=\bm{A}\hat{Y}=\langle\hat{Y}_0,\hat{Y}_1,\dots,\hat{Y}_{w^*}\rangle$.
Note that $A^*(Y)\tilde{W}\subseteq\tilde{W}$.
Following \cite[Definition 3.7]{KLM2009preA}, we call $Y$ \emph{Leonard} (\emph{with respect to} $\theta_1$) if the matrix representing $A^*(Y)|_{\tilde{W}}$ with respect to the basis $\{E_i\hat{Y}\}_{i=0}^{w^*}$ for $\tilde{W}$ is irreducible\footnote{A tridiagonal matrix is \emph{irreducible} \cite{Terwilliger2001LAA} if all the superdiagonal and subdiagonal entries are nonzero.} tridiagonal.
Set $\mathfrak{b}=A|_{\tilde{W}}$, $\mathfrak{b}^*=A^*(Y)|_{\tilde{W}}$, $\mathfrak{f}_i=E_i|_{\tilde{W}}$, $\mathfrak{f}_i^*=E_i^*(Y)|_{\tilde{W}}$ $(0\leqslant i\leqslant w^*)$.
Then $Y$ is Leonard if and only if $\Phi(\Gamma;Y):=\left(\mathfrak{b};\mathfrak{b}^*;\{\mathfrak{f}_i\}_{i=0}^{w^*};\{\mathfrak{f}_i^*\}_{i=0}^{w^*}\right)$ is a Leonard system.
The following is dual to \eqref{when Gamma_Y is distance-regular}:

\begin{prop}\label{when Y is Leonard}
Pick any $x\in Y$, and let the parameter array of $\Phi=\Phi(\Gamma;x)$ be given as in \eqref{list of parameter arrays}.
Suppose $w^*>1$.
Then $Y$ is Leonard (with respect to $\theta_1$) precisely for Cases I, IA, II, IIA, IIB, IIC; or Case III with $w$ even.
If this is the case then the bilinear form $(\cdot|\cdot):\bm{A}\hat{x}\times\bm{A}\hat{Y}\rightarrow\mathbb{C}$ defined by $(u|u')=u^{\mathsf{T}}u'$ is $0$-balanced with respect to $\Phi^*$, $\Phi(\Gamma;Y)^*$.
\end{prop}

\begin{proof}
Note that $E_i^*(x)E_j^*(Y)=0$ whenever $i<j$ or $i>j+w$ (cf.~\eqref{Hosoya--Suzuki observation}).
Hence if $Y$ is Leonard then $(\cdot|\cdot)$ is $0$-balanced with respect to $\Phi^*$, $\Phi(\Gamma;Y)^*$, so that by \eqref{characterization of bilinear form in parametric form} it follows that $w$ must be even if $p(\Phi)$ is of Case III.\footnote{The permutation $*$ (see footnote \ref{D4 action}) leaves each of Cases I, IA, II, IIC, III invariant and swaps Cases IIA and IIB.}
Conversely, suppose $p(\Phi)$ (and $w$) satisfies one of the cases mentioned in \eqref{when Y is Leonard}.
Then by \cite[Theorem 7.3]{Tanaka2009LAAb} there are operators $\mathfrak{c}$, $\mathfrak{c}^*$ on $\tilde{W}$ such that $\bigl(\mathfrak{c};\mathfrak{c}^*;\{\mathfrak{f}_i\}_{i=0}^{w^*};\{\mathfrak{f}_i^*\}_{i=0}^{w^*}\bigr)$ is a Leonard system.
Note that $\mathfrak{b}^*\in\langle \mathfrak{f}_0^*,\mathfrak{f}_1^*,\dots,\mathfrak{f}_{w^*}^*\rangle$, so that $\mathfrak{b}^*$ is a polynomial in $\mathfrak{c}^*$.
Since $\mathfrak{b}^*\mathfrak{f}_0\tilde{W}= \langle E_1\hat{Y} \rangle =\mathfrak{f}_1\tilde{W}$, it follows from \eqref{ith component of 0*} that $\mathfrak{b}^*=\xi^*\mathfrak{c}^*+\zeta^*\tilde{\mathfrak{1}}$ for some $\xi^*,\zeta^*\in\mathbb{C}$ with $\xi^*\ne 0$, where $\tilde{\mathfrak{1}}$ is the identity operator on $\tilde{W}$.
Hence the matrix representing $\mathfrak{b}^*$ with respect to $\{E_i\hat{Y}\}_{i=0}^{w^*}$ is irreducible tridiagonal.
In other words, $Y$ is Leonard, as desired.
\end{proof}

\begin{rem}
Suppose $\Gamma$ is a translation distance-regular graph \cite[\S 11.1C]{BCN1989B} and $Y$ is also a subgroup of the abelian group $X$.
Then by \cite[Proposition 3.3, Theorem 3.10]{KLM2009preA}, $Y$ is Leonard if and only if the coset graph $\Gamma/Y$ is $Q$-polynomial.
Hence \eqref{when Y is Leonard} strengthens \cite[Theorem 4]{BGKM2003JCTA}, which states that $\Gamma/Y$ is $Q$-polynomial if it is primitive.
Note that if $Y$ is Leonard then $\Phi(\Gamma/Y;Y)$ (where $Y$ is a \emph{vertex} of $\Gamma/Y$) is affine isomorphic to $\Phi(\Gamma;Y)$.
\end{rem}

It seems that \eqref{when Y is Leonard} also motivates further analysis of the Terwilliger algebra \emph{with respect to} $Y$ in the sense of Suzuki \cite{Suzuki2005JAC}; this will be discussed elsewhere.

\section{The bipartite case}\label{sec: bipartite case}

\begin{conv}\label{bipartite case}
With reference to \eqref{general convention}, in this section only we further assume that $\Gamma$ is bipartite and $d\geqslant 6$ (so that the halved graphs have diameter at least three).
\end{conv}

With reference to \eqref{bipartite case}, fix $x\in X$ and let $\Gamma_d^2=\Gamma_d^2(x)$ be the graph with vertex set $\Gamma_d=\Gamma_d(x)$ and edge set $\{(y,z)\in\Gamma_d\times\Gamma_d:\partial(y,z)=2\}$.
Caughman \cite[Theorems 9.2, 9.6, Corollary 4.4]{Caughman2003EJC} showed that $\Gamma_d^2$ is distance-regular and $Q$-polynomial with diameter $\mathfrak{d}$, where $\mathfrak{d}$ equals half the width of $\Gamma_d$.
In this section, we shall prove a result relating descendents of a halved graph of $\Gamma$ to those of $\Gamma_d^2$; see \eqref{intersection of descendent of bipartite half with last subconstituent} below.

Write $E_i^*=E_i^*(x)$ $(0\leqslant i\leqslant d)$ and $\bm{T}=\bm{T}(x)$.
Let $W$ be an irreducible $\bm{T}$-module with endpoint $\epsilon$, dual endpoint $\epsilon^*$ and diameter $\delta$ (see \eqref{Phi(W)}).
By \cite[Lemma 9.2, Theorem 9.4]{Caughman1999DM}, $W$ is thin, dual thin and $2\epsilon^*+\delta=d$.
In particular, $0\leqslant\epsilon^*\leqslant\lfloor d/2\rfloor$ and $\epsilon\leqslant 2\epsilon^*$.
Let $U_{ij}$ be the sum of the irreducible $\bm{T}$-modules $W$ in $\mathbb{C}^X$ with $\epsilon=i$ and $\epsilon^*=j$.
By \cite[Theorem 13.1]{Caughman1999DM}, the (nonzero) $U_{ij}$ are the homogeneous components of $\mathbb{C}^X$.
Note that $E_d^*U_{ij}=0$ unless $i=2j$, so that $E_d^*\mathbb{C}^X=\sum_{j=0}^{\lfloor d/2\rfloor}E_d^*U_{2j,j}$ (orthogonal direct sum).
By \cite[Theorem 9.2]{Caughman2003EJC} $E_d^*\bm{A}E_d^*$ gives the Bose--Mesner algebra of $\Gamma_d^2$, and each of $E_d^*U_{2j,j}$ $(0\leqslant j\leqslant \lfloor d/2\rfloor)$ is a (not necessarily maximal) eigenspace for $E_d^*\bm{A}E_d^*$.

We now compute the eigenvalue of the adjacency matrix $E_d^*A_2E_d^*$ of $\Gamma_d^2$ on $E_d^*U_{2j,j}$.
If $\Gamma$ is the $d$-cube $H(d,2)$ then $\Gamma_d$ is a singleton and there is nothing to discuss.
Suppose $\Gamma$ is the folded cube $\bar{H}(2d,2)$.
Let $W\subseteq U_{2j,j}$ and let $\Phi=\Phi(W)$ be as in \eqref{Phi(W)}.
By \cite[Example 6.1]{Terwilliger1993JACb}, $a_i(\Phi)=0$ $(0\leqslant i\leqslant \delta)$, $b_i(\Phi)=2\delta-i$ $(0\leqslant i\leqslant\delta-1)$, $c_i(\Phi)=i$ $(1\leqslant i\leqslant\delta-1)$ and $c_{\delta}(\Phi)=2\delta$, where $\delta=d-2j$.
Using $A^2=c_2A_2+kI$ we find that $E_d^*A_2E_d^*$ has eigenvalue $(d-2j)^2-2j$ on $E_d^*W$ (and hence on $E_d^*U_{2j,j}$).
Next suppose $\Gamma\ne H(d,2),\bar{H}(2d,2)$.
Then by \cite[pp.~89--91]{Caughman1999DM}, $p(\Phi(\Gamma))$ satisfies Case I in \eqref{list of parameter arrays} and in this case there are scalars $q,s^*\in\mathbb{C}$ (independent of $j$) such that $a_i(\Phi)=0$ $(0\leqslant i\leqslant\delta)$,
\begin{equation*}
	b_i(\Phi)=\frac{h(q^{\delta}-q^i)(1-s^*q^{4j+i+1})}{q^{\delta+j}(1-s^*q^{4j+2i+1})}, \quad c_i(\Phi)=\frac{h(q^i-1)(1-s^*q^{4j+\delta+i+1})}{q^{\delta+j}(1-s^*q^{4j+2i+1})}
\end{equation*}
for $1\leqslant i\leqslant \delta-1$, and $b_0(\Phi)=c_{\delta}(\Phi)=h(q^{-j}-q^{j-d})$, where
\begin{equation*}
	h=\frac{q^d(1-s^*q^3)}{(q-1)(1-s^*q^{d+2})}.
\end{equation*}
Likewise we find that the eigenvalue of $E_d^*A_2E_d^*$ on $E_d^*U_{2j,j}$ is given by
\begin{gather*}
	\frac{1-s^*q^5}{(q^2-1)(1-s^*q^{d+2})(1-s^*q^{d+3})(1-s^*q^{2d-1})q} \\
	\times\left((q^d-1)(q^d-q)(1-s^{*2}q^{2d+2})+\frac{q^{2d}(1-s^*q^3)(1-q^{2j})(1-s^*q^{2j})}{q^{2j}}\right).
\end{gather*}

\begin{prop}\label{intersection of descendent of bipartite half with last subconstituent}
Referring to \eqref{bipartite case}, suppose $\Gamma_d^2$ has diameter $\lfloor d/2\rfloor$.
If $Y$ is a descendent of a halved graph of $\Gamma$ with width $w$, then $Y\cap\Gamma_d$ is a descendent of $\Gamma_d^2$ with width $w$, provided that it is nonempty.
\end{prop}

\begin{proof}
Note that $Y\cap\Gamma_d$ has width at most $w$ in $\Gamma_d^2$, and that the characteristic vector of $Y\cap\Gamma_d$ is $E_d^*\hat{Y}$.
Let $w^*$ be the dual width of $Y$ in the halved graph.
Then by \cite[p.~328, Theorem 6.4]{BI1984B} we find $\hat{Y}\in\sum_{j=0}^{w^*}(E_j+E_{d-j})\mathbb{C}^X$, so that $E_d^*\hat{Y}\in\sum_{j=0}^{w^*}E_d^*U_{2j,j}$.
By assumption, $\Gamma\ne H(d,2)$.
Suppose $\Gamma=\bar{H}(2d,2)$.
Then $\Gamma_d^2$ is the folded Johnson graph $\bar{J}(2d,d)$.
It follows (e.g., from \cite[p.~301]{BI1984B} or \cite[Example 6.1]{Terwilliger1993JACb}) that $(d-2j)^2-2j$ is the $j^{\mathrm{th}}$ eigenvalue of $\bar{J}(2d,d)$ in the $Q$-polynomial ordering $(0\leqslant j\leqslant\lfloor d/2\rfloor)$.
Likewise, if $\Gamma\ne H(d,2),\bar{H}(2d,2)$, then by the data in \cite[p.~469]{Caughman2003EJC}, \cite[pp.~264--265]{BI1984B} we routinely find that the ordering $\{E_d^*U_{2j,j}\}_{j=0}^{\lfloor d/2\rfloor}$ of the eigenspaces of $\Gamma_d^2$ also agrees with the $Q$-polynomial ordering of $\Gamma_d^2$.
Hence it follows that $Y\cap\Gamma_d$ has dual width at most $w^*$ in $\Gamma_d^2$.
Since $\Gamma_d^2$ has diameter $\lfloor d/2\rfloor=w+w^*$, we find that $Y\cap\Gamma_d$ is a descendent of $\Gamma_d^2$ and has width $w$, as desired.
\end{proof}

\section{Convexity and graphs with classical parameters}\label{sec: convexity and classical parameters}

In this section, we shall prove our main results concerning convexity of descendents and classical parameters.
Let $\Phi$ be the Leonard system from \eqref{Leonard system}.
By \eqref{bi, ci} we find
\begin{equation}\label{normalization}
	\frac{b_i(\Phi)}{c_1(\Phi)}=\frac{\varphi_{i+1}\eta_d^*(\theta_0^*)\tau_i^*(\theta_i^*)}{\phi_1\eta_{d-1}^*(\theta_1^*)\tau_{i+1}^*(\theta_{i+1}^*)}, \quad \frac{c_i(\Phi)}{c_1(\Phi)}=\frac{\phi_i\eta_d^*(\theta_0^*)\eta_{d-i}^*(\theta_i^*)}{\phi_1\eta_{d-1}^*(\theta_1^*)\eta_{d-i+1}^*(\theta_{i-1}^*)} \quad (0\leqslant i\leqslant d).
\end{equation}
Note that the values in \eqref{normalization} are invariant under affine transformation of $\Phi$.\footnote{Therefore, if $p(\Phi)$ satisfies, say, Case I in \eqref{list of parameter arrays} then the resulting formulae involve $q,r_1,r_2,s,s^*$ only, and are independent of $h,h^*,\theta_0,\theta_0^*$.}
Moreover, if $\Phi=\Phi(\Gamma)$ then by \eqref{parameters of Phi and Gamma} they coincide with $b_i(\Gamma)$ and $c_i(\Gamma)$, respectively, since $c_1(\Gamma)=1$.
The following is a refinement of \cite[Theorem 8.4.1]{BCN1989B}, and is verified using \eqref{classical parameters}, \eqref{theta* in terms of classical parameters}, \eqref{normalization}:

\begin{prop}\label{cases and classical parameters}
With reference to \eqref{general convention}, let the parameter array of $\Phi=\Phi(\Gamma)$ be given as in \eqref{list of parameter arrays}.
Then $\Gamma$ has classical parameters if and only if $p(\Phi)$ satisfies either Case I and $s^*=0$; or Cases IA, IIA, IIC.
If this is the case then $p(\Phi)$ and the classical parameters $(d,q,\alpha,\beta)$ are related as follows:
\end{prop}
\begin{center}
\begin{tabular}{cccc}
\hline \\[-.18in]
Case & $q$ & $\alpha$ & $\beta$ \\
\hline \\[-.18in]
	I, $s^*=r_1=0$ & $q$ & $\dfrac{r_2(1-q)}{sq^d-r_2}$ & $\dfrac{r_2q-1}{q(sq^d-r_2)}$ \\[.14in]
	IA & $q$ & $\dfrac{r(1-q)}{sq^d-r}$ & $\dfrac{r}{sq^d-r}$ \\[.14in]
	IIA & $1$ & $\dfrac{1}{r-s-d}$ & $\dfrac{-1-r}{r-s-d}$ \\[.14in]
	IIC & $1$ & $0$ & $\dfrac{-r}{r-ss^*}$ \\[.12in]
\hline
\end{tabular}
\end{center}

\begin{thm}\label{convexity and classical parameters}
With reference to \eqref{general convention}, suppose $1<w<d$.
Then $Y$ is convex precisely when $\Gamma$ has classical parameters (with standard $Q$-polynomial ordering).
\end{thm}

\begin{proof}
Let $\Phi=\Phi(\Gamma)$ and let $p(\Phi)$ be given as in \eqref{list of parameter arrays}.
First, we may assume that $\Gamma_Y$ is distance-regular, or equivalently, by \eqref{when Gamma_Y is distance-regular}, that $w^*$ is even if $p(\Phi)$ satisfies Case III.
Indeed, if $Y$ is convex then $\Gamma_Y$ is connected and hence distance-regular by \eqref{connectivity implies distance-regularity}.
On the other hand, if $\Gamma$ has classical parameters then by \eqref{cases and classical parameters} $p(\Phi)$ does not satisfy Case III in the first place.
Let $\Phi'=\Phi(\Gamma_Y)$.
Then by \eqref{when Gamma_Y is distance-regular} $p(\Phi')$ takes the form given in \eqref{characterization of bilinear form in parametric form} with $\rho=0$.
Note that $Y$ is convex if and only if $c_i(\Gamma)=c_i(\Gamma_Y)$ for all $1\leqslant i\leqslant w$.
Using \eqref{normalization} we find that $c_i(\Gamma)/c_i(\Gamma_Y)$ equals
\begin{align*}
	& \frac{(1-s^*q^{w+2})(1-s^*q^{i+d+1})}{(1-s^*q^{d+2})(1-s^*q^{i+w+1})} && \text{for Case I}, \\
	& \frac{(s^*+w+2)(s^*+i+d+1)}{(s^*+d+2)(s^*+i+w+1)} && \text{for Cases II, IIB}, \\
	& \frac{-s^*+w+2}{-s^*+d+2} && \text{for Case III}, \ d \ \text{even}, \ w \ \text{even}, \ i \ \text{even}, \\
	& \frac{(-s^*+w+2)(-s^*+i+d+1)}{(-s^*+d+2)(-s^*+i+w+1)} && \text{for Case III}, \ d \ \text{even}, \ w \ \text{even}, \ i \ \text{odd}, \\
	& \frac{-s^*+i+d+1}{-s^*+i+w+1} && \text{for Case III}, \ d \ \text{odd}, \ w \ \text{odd}, \ i \ \text{even}, \\
	& 1 && \text{for Cases IA, IIA, IIC; or} \\
	&&& \text{Case III}, \ d \ \text{odd}, \ w \ \text{odd}, \ i \ \text{odd}.
\end{align*}
Since $1<w<d$ it follows that $Y$ is convex precisely when $p(\Phi)$ satisfies one of the following: Case I, $s^*=0$; or Cases IA, IIA, IIC.
Hence the result follows from \eqref{cases and classical parameters}.
\end{proof}

The following is another important consequence of \eqref{Phi(Gamma) and Phi(Gamma_Y) determine each other}, \eqref{cases and classical parameters}, \eqref{characterization of bilinear form in parametric form}:

\begin{thm}\label{descendents inherit classical parameters}
Given scalars $q$, $\alpha$ and $\beta$, if $\Gamma$ has classical parameters $(d,q,\alpha,\beta)$ (with standard $Q$-polynomial ordering) then $\Gamma_Y$ is distance-regular and has classical parameters $(w,q,\alpha,\beta)$.
The converse also holds, provided $w\geqslant 3$. 
\end{thm}

\begin{rem}
In \cite[Proposition 2]{Tanaka2006JCTA}, the author showed that $\iota(\Gamma_Y)$ is uniquely determined by $w$, and when $\Gamma$ is of ``semilattice-type'' the convexity of $Y$ was then derived from the existence of a specific example.
We may remark that \eqref{convexity and classical parameters}, \eqref{descendents inherit classical parameters} supersede these results and give an answer to the problem raised in \cite[p. 907, Remark]{Tanaka2006JCTA}.
\end{rem}

We end this section with a comment.
Recall that $Y$ is called \emph{strongly closed} if for any $x,y\in Y$ we have $\{z\in X:\partial(x,z)+\partial(z,y)\leqslant\partial(x,y)+1\}\subseteq Y$.

\begin{note}
With reference to \eqref{general convention}, suppose $1<w<d$.
Then $Y$ is strongly closed precisely when $\Gamma$ has classical parameters $(d,q,0,\beta)$ (with standard $Q$-polynomial ordering).
\end{note}

\begin{proof}
We may assume that $Y$ is convex, or equivalently, by \eqref{convexity and classical parameters}, that $\Gamma$ has classical parameters $(d,q,\alpha,\beta)$.
In particular, $c_i(\Gamma)=c_i(\Gamma_Y)$ $(1\leqslant i\leqslant w)$.
Note that $Y$ is then strongly closed if and only if $a_i(\Gamma)=a_i(\Gamma_Y)$ for all $1\leqslant i\leqslant w$.
Since $\Gamma_Y$ has classical parameters $(w,q,\alpha,\beta)$ by \eqref{descendents inherit classical parameters}, it follows from \eqref{classical parameters} that this happens precisely when $\alpha=0$, as desired.
\end{proof}

\section{Quantum matroids and descendents}\label{sec: quantum matroids}

There are many results on distance-regular graphs with the property that any pair of vertices $x,y$ is contained in a strongly closed subset with width $\partial(x,y)$; see e.g., \cite{Weng1998GC,Hiraki2009GC}.
In view of the results of \S \ref{sec: convexity and classical parameters}, in this section we assume $\Gamma$ has classical parameters and look at the implications of similar existence conditions on descendents.

First we recall some facts concerning quantum matroids \cite{Terwilliger1996ASPM} and related distance-regular graphs.
A finite nonempty poset $\mathscr{P}$ is a \emph{quantum matroid} if

\begin{enumerate}[(QM1)]
\setlength{\itemsep}{0cm}
\item $\mathscr{P}$ is ranked;
\item $\mathscr{P}$ is a (meet) semilattice;
\item For all $x\in\mathscr{P}$, the interval $[0,x]$ is a modular atomic lattice;
\item For all $x,y\in\mathscr{P}$ satisfying $\mathrm{rank}(x)<\mathrm{rank}(y)$, there is an atom $a\in\mathscr{P}$ such that $a\leqslant y$, $a\not\leqslant x$ and $x\vee a$ exists in $\mathscr{P}$.
\end{enumerate}

We say $\mathscr{P}$ is \emph{nontrivial} if $\mathscr{P}$ has rank $d\geqslant 2$ and is not a modular atomic lattice.
Suppose $\mathscr{P}$ is nontrivial.
Then $\mathscr{P}$ is called $q$-\emph{line regular} if each rank $2$ element covers exactly $q+1$ elements; $\mathscr{P}$ is $\beta$-\emph{dual-line regular} if each element with rank $d-1$ is covered by exactly $\beta+1$ elements; $\mathscr{P}$ is $\alpha$-\emph{zig-zag regular} if for all pairs $(x,y)$ such that $\mathrm{rank}(x)=d-1$, $\mathrm{rank}(y)=d$ and $x$ covers $x\wedge y$, there are exactly $\alpha+1$ pairs $(x_1,y_1)$ such that $y_1$ covers both $x$ and $x_1$, and $y$ covers $x_1$.
We say $\mathscr{P}$ is \emph{regular} if $\mathscr{P}$ is line regular, dual-line regular and zig-zag regular.
Suppose $\mathscr{P}$ is nontrivial and regular with parameters $(d,q,\alpha,\beta)$.
Let $\mathrm{top}(\mathscr{P})$ be the top fiber of $\mathscr{P}$ and set $R=\left\{(x,y)\in\mathrm{top}(\mathscr{P})\times\mathrm{top}(\mathscr{P}):x,y\ \text{cover}\ x\wedge y\right\}$.
Then by \cite[Theorem 38.2]{Terwilliger1996ASPM}, $\Gamma=(\mathrm{top}(\mathscr{P}),R)$ is distance-regular.
Moreover, it has classical parameters $(d,q,\alpha,\beta)$ provided that the diameter equals $d$.

We now list five examples of nontrivial regular quantum matroids.
In \eqref{truncated Boolean algebra}--\eqref{polar spaces} below, the partial order on $\mathscr{P}$ will always be defined by inclusion and $Y$ will denote a nontrivial descendent of $\Gamma$.

\begin{exam}\label{truncated Boolean algebra}
The truncated Boolean algebra $B(d,\nu)$ $(\nu >d)$.
Let $\Omega$ be a set of size $\nu$ and $\mathscr{P}=\{x \subseteq \Omega :|x| \leqslant d\}$.
$\mathscr{P}$ has parameters $(d,1,1,\nu-d)$ and $\mathrm{top}(\mathscr{P})$ induces the Johnson graph $J(\nu,d)$ \cite[\S 9.1]{BCN1989B}.
If $\nu \geqslant 2d$ then $Y$ satisfies one of the following:
(i) $Y=\{x \in \mathrm{top}(\mathscr{P}):u \subseteq x\}$ for some $u \subseteq \Omega$ with $|u|=w^*$;
(ii) $\nu=2d$ and $Y=\{x \in \mathrm{top}(\mathscr{P}):x \subseteq u\}$ for some $u \subseteq \Omega$ with $|u|=d+w$.
\end{exam}

\begin{exam}\label{Hamming matroid}
The Hamming matroid $H(d,\ell)$ $(\ell \geqslant 2)$.
Let $\Omega_1,\Omega_2,\dots,\Omega_d$ be pairwise disjoint sets of size $\ell$, $\Omega=\bigcup_{i=1}^d\Omega_i$ and $\mathscr{P}=\{x \subseteq \Omega :|x \cap \Omega_i| \leqslant 1\ (1 \leqslant i \leqslant d)\}$.
$\mathscr{P}$ has parameters $(d,1,0,\ell-1)$ and $\mathrm{top}(\mathscr{P})$ induces the Hamming graph $H(d,\ell)$ \cite[\S 9.2]{BCN1989B}.
$Y$ is of the form $\{x\in\mathrm{top}(\mathscr{P}):u\subseteq x\}$ for some $u\in\mathscr{P}$ with $|u|=w^*$.
\end{exam}

\begin{exam}\label{truncated projective geometry}
The truncated projective geometry $L_q(d,\nu)$ $(\nu>d)$.
Let $\mathscr{P}$ be the set of subspaces $x$ of $\mathbb{F}_q^{\nu}$ with $\dim x \leqslant d$.
$\mathscr{P}$ has parameters $(d,q,q,\beta)$ where $\beta+1=\gauss{\nu-d+1}{1}_q$, and $\mathrm{top}(\mathscr{P})$ induces the Grassmann graph $J_q(\nu,d)$ \cite[\S 9.3]{BCN1989B}.
If $\nu\geqslant 2d$ then $Y$ satisfies one of the following:
(i) $Y=\{x \in \mathrm{top}(\mathscr{P}): u \subseteq x\}$ for some subspace $u \subseteq \mathbb{F}_q^{\nu}$ with $\dim u=w^*$;
(ii) $\nu=2d$ and $Y=\{x \in \mathrm{top}(\mathscr{P}): x \subseteq u\}$ for some subspace $u \subseteq \mathbb{F}_q^{\nu}$ with $\dim u=d+w$.
\end{exam}

\begin{exam}\label{attenuated space}
The attenuated space $A_q(d,d+e)$ $(e\geqslant 1)$.
Fix a subspace $E$ of $\mathbb{F}_q^{d+e}$ with $\dim E=e$, and let $\mathscr{P}$ be the set of subspaces $x$ of $\mathbb{F}_q^{d+e}$ with $x\cap E=0$.
$\mathscr{P}$ has parameters $(d,q,q-1,q^e-1)$, and $\mathrm{top}(\mathscr{P})$ induces the bilinear forms graph $\mathrm{Bil}_q(d,e)$ \cite[\S 9.5A]{BCN1989B}.
If $d \leqslant e$ then $Y$ satisfies one of the following:
(i) $Y=\{x \in \mathrm{top}(\mathscr{P}): u \subseteq x\}$ for some subspace $u \subseteq \mathbb{F}_q^{d+e}$ with $\dim u=w^*$ and $u\cap E=0$;
(ii) $d=e$ and $Y=\{x \in \mathrm{top}(\mathscr{P}): x \subseteq u\}$ for some subspace $u \subseteq \mathbb{F}_q^{d+e}$ with $\dim u=d+w$ and $\dim u \cap E=w$.
\end{exam}

\begin{exam}\label{polar spaces}
The classical polar spaces.
Let $V$ be one of the following spaces over $\mathbb{F}_q$ equipped with a nondegenerate form:
\begin{center}
\begin{tabular}{cclc}
\hline
\\[-.18in]
Name & $\dim V$ & \multicolumn{1}{c}{Form} & $e$ \\
\hline
\\[-.18in]
$[C_d(q)]$ & $2d$ & alternating & $1$ \\[.02in]
$[B_d(q)]$ & $2d+1$ & quadratic & $1$ \\[.02in]
$[D_d(q)]$ & $2d$ & quadratic (Witt index $d$) & $0$ \\[.02in]
$[^2\!D_{d+1}(q)]$ & $2d+2$ & quadratic (Witt index $d$) & $2$ \\[.02in]
$[^2\!A_{2d}(\ell)]$ & $2d+1$ & Hermitean ($q=\ell^2$) & $\frac{3}{2}$ \\[.02in]
$[^2\!A_{2d-1}(\ell)]$ & $2d$ & Hermitean ($q=\ell^2$) & $\frac{1}{2}$ \\[.02in]
\hline
\end{tabular}
\end{center}
Let $\mathscr{P}$ be the set of isotropic subspaces of $V$.
$\mathscr{P}$ has parameters $(d,q,0,q^e)$ and $\mathrm{top}(\mathscr{P})$ induces the dual polar graph on $V$ \cite[\S 9.4]{BCN1989B}.
$Y$ is of the form $\{x\in\mathrm{top}(\mathscr{P}):u\subseteq x\}$ for some $u\in\mathscr{P}$ with $\dim u=w^*$.
\end{exam}

We recall the following isomorphisms: $J(\nu,d)\cong J(\nu,\nu-d)$, $J_q(\nu,d) \cong J_q(\nu,\nu-d)$, $\mathrm{Bil}_q(d,e) \cong \mathrm{Bil}_q(e,d)$.
Note also that, in each of \eqref{truncated Boolean algebra}--\eqref{polar spaces}, the descendents with any fixed width form a single orbit under the full automorphism group of $\Gamma$.

\begin{thm}[{\cite[Theorem 39.6]{Terwilliger1996ASPM}}]\label{classification of regular quantum matroids}
Every nontrivial regular quantum matroid with rank at least four is isomorphic to one of \eqref{truncated Boolean algebra}--\eqref{polar spaces}.
\end{thm}

Now we return to the general situation \eqref{general convention}.
Let $\mathscr{P}$ be a nonempty family of descendents of $\Gamma$.
We say $\mathscr{P}$ \emph{satisfies (UD)$_i$} if any two vertices $x,y\in X$ at distance $i$ are contained in a unique descendent in $\mathscr{P}$, denoted $Y(x,y)$, with width $i$.
We shall assume the following three conditions until \eqref{when P is the whole set}:

\begin{conv}\label{Gamma has classical parameters}
$\Gamma$ has classical parameters $(d,q,\alpha,\beta)$;
\end{conv}

\begin{conv}\label{P satisfies (UD)}
$\mathscr{P}$ satisfies (UD)$_i$ for all $i$;
\end{conv}

\begin{conv}\label{P is closed under intersection}
$Y_1\cap Y_2\in\mathscr{P}$ for all $Y_1,Y_2\in\mathscr{P}$ such that $Y_1\cap Y_2\ne\emptyset$.
\end{conv}

Referring to \eqref{Gamma has classical parameters}--\eqref{P is closed under intersection}, define a partial order $\leqslant$ on $\mathscr{P}$ by reverse inclusion.
Our goal is to show that $\mathscr{P}$ is a nontrivial regular quantum matroid.
Note that $X$ is the minimal element of $\mathscr{P}$ and the maximal elements of $\mathscr{P}$ are precisely the singletons.
We shall freely use \eqref{fundamental inequality}, \eqref{transitivity}, \eqref{convexity and classical parameters}, \eqref{descendents inherit classical parameters}.
In particular, note that every $Y\in\mathscr{P}$ is convex, $\Gamma_Y$ is distance-regular with classical parameters $(w(Y),q,\alpha,\beta)$, and $\mathscr{P}_Y:=\{Z\in\mathscr{P}:Z\subseteq Y\}$ is a family of descendents of $\Gamma_Y$.

\begin{note}\label{descendents containing given two vertices}
For $0\leqslant i\leqslant j\leqslant d$, we have $|\{Y\in\mathscr{P}:x,y\in Y,\ w(Y)=j\}|=\gauss{d-i}{j-i}_q$ for any two vertices $x,y\in X$ with $\partial(x,y)=i$.
In particular, $q$ is a positive integer.
\end{note}

\begin{proof}
Count in two ways the sequences $(z_1,\dots,z_{j-i},Y)$ such that $z_l\in\Gamma_{i+l}(x)\cap\Gamma(z_{l-1})$ $(1\leqslant l\leqslant j-i)$ and $Y=Y(x,z_{j-i})$, where $z_0=y$.
\end{proof}

\begin{note}\label{descendents satisfy (UD)}
$\mathscr{P}_Y$ satisfies (UD)$_i$ in $\Gamma_Y$ for all $Y\in\mathscr{P}$ and $i$.
\end{note}

\begin{proof}
Let $x,y\in Y$ and set $Z=Y(x,y)\in\mathscr{P}$.
Then $Y\cap Z$ belongs to $\mathscr{P}_Y$ and contains $x$ and $y$, so that $Z=Y\cap Z\in\mathscr{P}_Y$.
\end{proof}

\begin{note}\label{descendents containing given one}
For $0\leqslant i\leqslant j\leqslant d$, we have $|\{Z\in\mathscr{P}:Y\subseteq Z,\ w(Z)=j\}|=\gauss{d-i}{j-i}_q$ for every $Y\in\mathscr{P}$ with $w(Y)=i$.
\end{note}

\begin{proof}
Pick $x,y\in Y$ with $\partial(x,y)=i$ and let $Z\in\mathscr{P}$.
Then by \eqref{descendents satisfy (UD)} we find $x,y\in Z$ if and only if $Y\subseteq Z$.
Hence the result follows from \eqref{descendents containing given two vertices}.
\end{proof}

\begin{note}\label{P satisfies (QM1)}
$\mathscr{P}$ satisfies (QM1) and $\mathrm{rank}(Y)=w^*(Y)$ for every $Y\in\mathscr{P}$.
\end{note}

\begin{proof}
Pick $Y,Z\in\mathscr{P}$ with $Y\subseteq Z$.
Applying \eqref{descendents containing given one} to $(\Gamma_Z,\mathscr{P}_Z)$, we find $Y$ covers $Z$ precisely when $w(Z)=w(Y)+1$, or equivalently, $w^*(Z)=w^*(Y)-1$.
Hence $\mathrm{rank}(Y)$ is well defined and equals $w^*(Y)$.
\end{proof}

\begin{note}\label{P satisfies (QM2)}
$\mathscr{P}$ satisfies (QM2).
Moreover, $w(Y_1\wedge Y_2)=w(Y_1\cup Y_2)$ for any $Y_1,Y_2\in\mathscr{P}$.
\end{note}

\begin{proof}
Let $Y_1,Y_2\in\mathscr{P}$ and suppose $Y_1\not\subseteq Y_2$, $Y_2\not\subseteq Y_1$.
Since $Y_1,Y_2$ are completely regular, for every $x\in X$ and $i\in\{1,2\}$ we have $w(\{x\}\cup Y_i)=\partial(x,Y_i)+w(Y_i)$.
Hence there are $x_1\in Y_1$, $x_2\in Y_2$ such that $\partial(x_1,x_2)=w(Y_1\cup Y_2)$.
Set $Z=Y(x_1,x_2)$.
Pick $y_1\in Y_1$ with $\partial(x_2,y_1)=\partial(x_2,Y_1)=\partial(x_1,x_2)-w(Y_1)$.
Then $\partial(x_1,y_1)=w(Y_1)$, so that $y_1\in Z$ and $Y_1=Y(x_1,y_1)\subseteq Z$ by \eqref{descendents satisfy (UD)}.
Likewise, $Y_2\subseteq Z$.
Hence $Z$ is a lower bound for $Y_1,Y_2$.
Any lower bound for $Y_1,Y_2$ contains $x_1,x_2$, and thus $Z$ by \eqref{descendents satisfy (UD)}, whence $Z=Y_1\wedge Y_2$.
\end{proof}

\begin{note}
$\mathscr{P}$ satisfies (QM3).
\end{note}

\begin{proof}
Let $Y\in\mathscr{P}$.
Then the interval $[X,Y]$ is a lattice since $\mathscr{P}$ satisfies (QM2) by \eqref{P satisfies (QM2)}.
By \eqref{descendents containing given one}, every $Z\in [X,Y]$ with $w^*(Z)\geqslant 2$ covers $\gauss{w^*(Z)}{1}_q\geqslant 2$ elements in $\mathscr{P}$, so that $[X,Y]$ is atomic.
Let $Y_1,Y_2\in [X,Y]$ be distinct and suppose $w^*(Y_1)=w^*(Y_2)$.
Set $i=w(Y_1)$, $j=w^*(Y_1)$ for brevity.
We claim $w^*(Y_1\wedge Y_2)=j-1$ if and only if $w^*(Y_1\cap Y_2)=j+1$, where we recall $Y_1\vee Y_2=Y_1\cap Y_2$.
First suppose $w^*(Y_1\cap Y_2)=j+1$.
Then $w(Y_1\cap Y_2)=i-1$ and by \eqref{transitivity}, \eqref{fundamental inequality} we find $Y_1\cap Y_2$ has covering radius one in each of $\Gamma_{Y_1}$, $\Gamma_{Y_2}$, so that $w(Y_1\cup Y_2)\leqslant w(Y_1\cap Y_2)+2=i+1$.
Since $w(Y_1\wedge Y_2)>i$, we find $w(Y_1\wedge Y_2)=i+1$ by \eqref{P satisfies (QM2)} and thus $w^*(Y_1\wedge Y_2)=j-1$.
Next suppose $w^*(Y_1\wedge Y_2)=j-1$.
Then by \eqref{descendents containing given one} and since $\gauss{j}{1}_q>\gauss{j-1}{1}_q$, there is $C\in\mathscr{P}$ such that $w^*(C)=1$, $Y_1\subseteq C$, $Y_1\wedge Y_2\not\subseteq C$.
We have $Y_1=C\cap(Y_1\wedge Y_2)$ and hence $Y_1\cap Y_2=C\cap Y_2$.
Since $\hat{C}\circ\hat{Y_2}\in\sum_{l=0}^{j+1}E_l\mathbb{C}^X$ by virtue of \eqref{Norton}, we find $w^*(Y_1\cap Y_2)\leqslant j+1$.
By \eqref{fundamental inequality} and since $w(Y_1\cap Y_2)<i$, it follows that $w^*(Y_1\cap Y_2)=j+1$.
The claim now follows, and therefore $[X,Y]$ is modular.
\end{proof}

\begin{note}
$\mathscr{P}$ satisfies (QM4).
\end{note}

\begin{proof}
Let $Y_1,Y_2\in\mathscr{P}$ and suppose $w^*(Y_1)<w^*(Y_2)$.
First assume $Y_1\cap Y_2\ne\emptyset$.
Then by \eqref{descendents containing given one} and since $\gauss{w^*(Y_2)}{1}_q>\gauss{w^*(Y_1)}{1}_q$, there is $C\in\mathscr{P}$ such that $w^*(C)=1$, $Y_2\subseteq C$, $Y_1\not\subseteq C$.
Since $Y_1\cap C\ne\emptyset$ we find $Y_1\vee C=Y_1\cap C$.
Next assume $Y_1\cap Y_2=\emptyset$, and pick $x_1,y_1\in Y_1$ and $x_2\in Y_2$ such that $\partial(x_1,x_2)=w(Y_1\cup Y_2)$ and $\partial(x_2,y_1)=\partial(x_2,Y_1)=\partial(x_1,x_2)-w(Y_1)$ as in the proof of \eqref{P satisfies (QM2)}.
Set $Z_1=Y_1\wedge Y_2$, $Z_2=\{y_1\}\wedge Y_2$.
Then $Z_2\subseteq Z_1$, and since $w(\{y_1\}\cup Y_2)=\partial(y_1,Y_2)+w(Y_2)\leqslant\partial(y_1,x_2)+w(Y_2)=\partial(x_1,x_2)-w(Y_1)+w(Y_2)<w(Y_1\cup Y_2)$, it follows from \eqref{P satisfies (QM2)} that $w(Z_1)>w(Z_2)$, i.e., $w^*(Z_1)<w^*(Z_2)$.
Again there is $C\in\mathscr{P}$ such that $w^*(C)=1$, $Z_2\subseteq C$, $Z_1\not\subseteq C$.
Note that $Y_2\subseteq C$ and $Y_1\not\subseteq C$.
Since $y_1\in C$, we find $Y_1\cap C\ne\emptyset$ and $Y_1\vee C=Y_1\cap C$, as desired.
\end{proof}

\begin{note}
$\mathscr{P}$ is $q$-line regular, $\beta$-dual-line regular and $\alpha$-zig-zag regular.
\end{note}

\begin{proof}
By \eqref{descendents containing given one}, $\mathscr{P}$ is $q$-line regular.
Pick any $Y\in\mathscr{P}$ with $w(Y)=1$.
Then $|Y|=\beta+1$ by \eqref{descendents inherit classical parameters}, \eqref{classical parameters}, so that $\mathscr{P}$ is $\beta$-dual-line regular.
Let $x\in X$ and suppose $w(\{x\}\wedge Y)=2$, i.e., $\partial(x,Y)=1$.
We count the pairs $(y,Z)\in X\times\mathscr{P}$ such that $w(Z)=1$, $y\in Y$, $y\in Z$, $x\in Z$.
Note that $x,y$ must be adjacent and $Z=Y(x,y)$.
Hence the number of such pairs is $|\Gamma(x)\cap Y|$.
By \eqref{descendents inherit classical parameters}, \eqref{classical parameters}, the strongly regular graph $\Delta=\Gamma_{\{x\}\wedge Y}$ satisfies $k(\Delta)=\beta(q+1)$, $a_1(\Delta)=\beta+\alpha q-1$, $c_2(\Delta)=(\alpha+1)(q+1)$, and hence (\cite[Theorem 1.3.1]{BCN1989B}) has smallest eigenvalue $\theta_2(\Delta)=-(q+1)$, so that $Y$ attains the Hoffman bound $1-k(\Delta)/\theta_2(\Delta)=\beta+1$.
By \cite[Proposition 1.3.2]{BCN1989B}, $|\Gamma(x)\cap Y|=-c_2(\Delta)/\theta_2(\Delta)=\alpha+1$ and therefore $\mathscr{P}$ is $\alpha$-zig-zag regular.
\end{proof}

To summarize:

\begin{prop}
Referring to \eqref{Gamma has classical parameters}--\eqref{P is closed under intersection}, $\mathscr{P}$ is a nontrivial regular quantum matroid with parameters $(d,q,\alpha,\beta)$.
\end{prop}

Hence it follows from \eqref{classification of regular quantum matroids} that

\begin{thm}\label{characterization of graphs with semilattice-type}
Referring to \eqref{Gamma has classical parameters}--\eqref{P is closed under intersection}, suppose further $d\geqslant 4$.
Then $\Gamma$ is either a Johnson, Hamming, Grassmann, bilinear forms, or dual polar graph.
\end{thm}

One may wish to modify the conditions \eqref{Gamma has classical parameters}--\eqref{P is closed under intersection} so that we still obtain a characterization similar to \eqref{characterization of graphs with semilattice-type}.
For example, it seems natural to assume that $\mathscr{P}$ is the set of \emph{all} descendents of $\Gamma$.
We now show that in this case \eqref{P is closed under intersection} is redundant.

\begin{prop}\label{when P is the whole set}
Suppose $\mathscr{P}$ is the set of descendents of $\Gamma$.
Then \eqref{Gamma has classical parameters}, \eqref{P satisfies (UD)} together imply \eqref{P is closed under intersection}.
\end{prop}

\begin{proof}
Note that \eqref{descendents containing given two vertices} holds without change.
We next show \eqref{descendents satisfy (UD)}.
Let $i$ be given.
By induction, assume that $\mathscr{P}_Z$ satisfies (UD)$_l$ in $\Gamma_Z$ for all $Z\in\mathscr{P}$ and $l>i$.
Let $Y\in\mathscr{P}$ and pick any $x,y\in Y$ with $\partial(x,y)=i$.
Since $\mathscr{P}_Y$ satisfies (UD)$_{i+1}$ in $\Gamma_Y$, for any $z\in\Gamma_{i+1}(x)\cap\Gamma(y)\cap Y$ we have $x,y\in Y(x,z)\subseteq Y$, so that by replacing $Y$ with $Y(x,z)$ we may assume $w(Y)=i+1$.
Since $\gauss{d-i}{1}_q>\gauss{d-i-1}{1}_q$, it follows from \eqref{descendents containing given two vertices} that there is $C\in\mathscr{P}$ with $x,y\in C$, $Y\not\subseteq C$, $w^*(C)=1$.
Set $Z=Y\cap C$.
Suppose $w(Z)=i+1$ and pick $u,v\in Z$ with $\partial(u,v)=i+1$.
Then $Y=Y(u,v)$.
But since $u,v\in C$ and $\mathscr{P}_C$ satisfies (UD)$_{i+1}$, we would have $Y\subseteq C$, a contradiction.
Hence $w(Z)=i$.
Moreover, $\hat{Z}=\hat{Y}\circ\hat{C}\in\sum_{l=0}^{d-i}E_l\mathbb{C}^X$ by virtue of \eqref{Norton}, whence $w^*(Z)\leqslant d-i$.
By \eqref{fundamental inequality} it follows that $Z\in\mathscr{P}$ and therefore $Y(x,y)=Z\subseteq Y$.
Hence $\mathscr{P}_Y$ satisfies (UD)$_i$ in $\Gamma_Y$.
We have now shown \eqref{descendents satisfy (UD)}.
Finally, we show \eqref{P is closed under intersection}.
Let $Y_1,Y_2\in\mathscr{P}$ and suppose $N:=Y_1\cap Y_2\ne\emptyset$.
Pick $x,y\in N$ such that $\partial(x,y)=w(N)$.
Then $Y(x,y)\subseteq N$ by \eqref{descendents satisfy (UD)}.
Since $Y(x,y)$ is completely regular and has width $w(N)$, for every $z\in N$ we have $w(N)\geqslant w(\{z\}\cup Y(x,y))=\partial(z,Y(x,y))+w(N)$, so that $z\in Y(x,y)$ and thus $N=Y(x,y)\in\mathscr{P}$, as desired.
\end{proof}

\begin{rem}
It should be remarked that $J(2d,d)$, $J_q(2d,d)$, $\mathrm{Bil}_q(d,d)$ are the only examples among the graphs listed in \eqref{characterization of graphs with semilattice-type} which do \emph{not} possess the property that ``the set of all descendents satisfies (UD)$_i$ for all $i$.''
This property seems particularly strong, so that it is a reasonable guess that \eqref{Gamma has classical parameters} would also be redundant.
\end{rem}

\section{Classifications}\label{sec: classifications}

In \S \ref{sec: quantum matroids}, we focused on the $5$ families of distance-regular graphs associated with short regular semilattices (or nontrivial regular quantum matroids).
In this section, we shall extend the classification of descendents to all of the $15$ known infinite families with classical parameters and with unbounded diameter.
We shall freely use \eqref{fundamental inequality}, \eqref{convexity and classical parameters}, \eqref{descendents inherit classical parameters}.

\begin{conv}
Throughout this section, $Y$ shall always denote a nontrivial descendent of $\Gamma$ with width $w$.
Descriptions of some of the graphs below involve $n\in\{2d,2d+1\}$, in which cases we use the following notation:
\begin{equation*}
	m=\begin{cases} 2d-1 & \text{if} \ n=2d, \\ 2d+1 & \text{if} \ n=2d+1.\end{cases}
\end{equation*}
\end{conv}

\subsection*{Doob graphs}

Let $d_1,d_2$ be positive integers, and let $\Gamma=\mathrm{Doob}(d_1,d_2):=\Gamma^1\times\Gamma^2\times\dots\times\Gamma^{d_1+d_2}$ be a Doob graph \cite[\S 9.2B]{BCN1989B}, where $\Gamma^1,\dots,\Gamma^{d_1}$ are copies of the Shrikhande graph $S$ on $16$ vertices and $\Gamma^{d_1+1},\dots,\Gamma^{d_1+d_2}$ are copies of the complete graph $K_4$ on $4$ vertices.
$\Gamma$ has classical parameters $(d,1,0,3)$, where $d=2d_1+d_2$.
In particular, $\iota(\Gamma)=\iota(H(d,4))$.

$\Gamma_Y$ has classical parameters $(w,1,0,3)$, so that $|Y|=4^w$.
Observe that the convex subsets of $X$ are precisely the direct products $Y^1\times Y^2\times\dots\times Y^{d_1+d_2}$ where $Y^1,\dots,Y^{d_1}$ are convex subsets of $V\!S$ and $Y^{d_1+1},\dots,Y^{d_1+d_2}$ are nonempty subsets of $V\!K_4$ (cf.~\cite[Proposition 5.11]{Lambeck1990D}).
Since $S$ has clique number three, it follows that
\begin{thm}\label{Doob graph}
Every descendent of $\mathrm{Doob}(d_1,d_2)$ is of the form $Y=Y^1\times Y^2\times\dots\times Y^{d_1+d_2}$, where either $Y^i = V \Gamma^i$ or $|Y^i|=1$, for each $i$ $(1\leqslant i\leqslant d_1+d_2)$.
\end{thm}

\subsection*{Halved cubes}
Let $\Gamma=\frac{1}{2}H(n,2)$ $(n\in\{2d,2d+1\})$ be a halved cube \cite[\S 9.2D]{BCN1989B} with vertex set $X=X_n^{\varepsilon}$, where $\varepsilon\in\mathbb{F}_2$ and $X_n^0$ (resp. $X_n^1$) is the set of even- (resp. odd-)weight vectors of $\mathbb{F}_2^n$.
$\Gamma$ has classical parameters $(d,1,2,m)$.

For every $x\in\mathbb{F}_2^n$ let $\mathrm{wt}(x)$ be its (Hamming) weight.
First suppose $w>1$.
Fix $x,y\in Y$ with $\partial(x,y)=w$.
For simplicity of notation, we assume $x=(x_1,t)$, $y=(y_1,t)$, where $x_1,y_1\in X_{2w}^{\varepsilon_1}$, $t\in X_{n-2w}^{\varepsilon_2}$ ($\varepsilon_1,\varepsilon_2\in\mathbb{F}_2$, $\varepsilon_1+\varepsilon_2=\varepsilon$) and $\mathrm{wt}(x_1-y_1)=2w$.
Note that $X_{2w}^{\varepsilon_1}\times\{t\}\subseteq Y$ by convexity.
Let $z=(z_1,z_2)\in Y\backslash (X_{2w}^{\varepsilon_1}\times\{t\})$.
Then it follows that $\mathrm{wt}(z_2-t)=1$.
Moreover, for every $u=(u_1,u_2)\in Y\backslash (X_{2w}^{\varepsilon_1}\times\{t\})$ we have $u_2=z_2$, for otherwise by convexity and since $w>1$ there would be a vector $v=(v_1,v_2)\in Y$ such that $\mathrm{wt}(v_2-t)=2$.
It follows that $\Gamma_Y$ has valency at most $\binom{2w+1}{2}=w(2w+1)$; but this is smaller than the expected valency $wm$ except when $n=2d$ and $w=d-1$, in which case we must have $Y=(X_{2w}^{\varepsilon_1}\times\{t\})\cup(X_{2w}^{1+\varepsilon_1}\times\{z_2\})$.

Next suppose $w=1$.
Then $|Y|=m+1$.
On the other hand, $\frac{1}{2}H(n,2)$ has maximal clique sizes $4$ and $n$ \cite[Theorem 14]{Hemmeter1986EJC}, from which it follows that $n=2d$.

\begin{thm}\label{halved cube}
Let $Y$ be a nontrivial descendent of $\frac{1}{2}H(n,2)$.
Then $n=2d$ and one of the following holds:
(i) $w=1$ and $Y=\{x\in X:\mathrm{wt}(x-z)=1\}$ for some $z\in X_n^{1+\varepsilon}=\mathbb{F}_2^n\backslash X$;
(ii)~$w=d-1$ and there are $a\in\mathbb{F}_2$ and $i\in\{1,2,\dots, n\}$ such that $Y=\{x=(\xi_1,\xi_2,\dots,\xi_n)\in X:\xi_i=a\}$.
\end{thm}

\subsection*{Hemmeter graphs}

Let $q$ be an \emph{odd} prime power and $\Sigma$ the dual polar graph on $[C_{d-1}(q)]$.
The Hemmeter graph $\Gamma=\mathrm{Hem}_d(q)$ \cite[\S 9.4C]{BCN1989B} is the extended bipartite double of $\Sigma$, so that $\Gamma$ has vertex set $X=X^+ \cup X^-$, where $X^{\pm}=\{x^{\pm}:x \in V\Sigma\}$ are two copies of $V\Sigma$.
$\Gamma$ has classical parameters $(d,q,0,1)$, which coincide with those of the dual polar graph on $[D_d(q)]$.

If $w=1$ then $Y$ is an edge.
Suppose $w>1$.
Let $\mathring{Y}=\{x \in V \Sigma :x^+\in Y\ \text{or}\ x^-\in Y\}$.
Then $\mathring{Y}$ is a convex subset of $V\Sigma$ with width $\mathring{w}\in\{w,w-1\}$.
By \cite[Proposition 5.19]{Lambeck1990D} and \cite[Lemma 10]{Hemmeter1986EJC}, there is an isotropic subspace $u$ of $[C_{d-1}(q)]$ with $\dim u=d-1-\mathring{w}$ such that $\mathring{Y}=I(u):=\{x \in V \Sigma :u \subseteq x\}$ if $\mathring{w}>1$, and $\mathring{Y}\subseteq I(u)$ if $\mathring{w}=1$.
Note that $|Y|=2\prod_{i=1}^{w-1}(q^i+1)$ and $|I(u)|=\prod_{i=1}^{\mathring{w}}(q^i+1)$ (cf.~\cite[Lemma 9.4.1]{BCN1989B}).
If $\mathring{w}=w(>1)$ then $|Y|<|I(u)|=|\mathring{Y}|$, a contradiction.
Hence $\mathring{w}=w-1$.
It follows that

\begin{thm}\label{Hemmeter graph}
Let $Y$ be a nontrivial descendent of $\mathrm{Hem}_d(q)$.
If $w=1$ then $Y$ is an edge.
If $w>1$ then there is an isotropic subspace $u$ of $[C_{d-1}(q)]$ with $\dim u=w^*$ such that $Y=\{x^+,x^-:x \in V \Sigma,\, u\subseteq x\}$.
\end{thm}

\subsection*{Hermitean forms graphs}

In this subsection and the next, we realize $\Gamma$ as ``affine subspaces'' of dual polar graphs.
See \cite[\S 9.5E]{BCN1989B} for the details.

Let $\ell$ be a prime power and $\Delta$ the dual polar graph on $[^2\!A_{2d-1}(\ell)]$.
Fix a vertex $z$ of $\Delta$.
The subgraph $\Gamma=\mathrm{Her}(d,\ell)$ of $\Delta$ induced on $X:=\Delta_d(z)$ is the Hermitean forms graph \cite[\S 9.5C]{BCN1989B}.
$\Gamma$ has classical parameters $(d,-\ell,-\ell-1,-(-\ell)^d-1)$.

By \cite[Proposition 5.30]{Lambeck1990D}, every noncomplete convex subset of $X$ is either of the form $I(u)=\{x\in X:u\subseteq x\}$ where $u$ is an isotropic subspace of $[^2\!A_{2d-1}(\ell)]$ with $\dim u\leqslant d-2$ and $z\cap u=0$, or isomorphic to the $4$-cycle $K_{2,2}$.
The latter case occurs only when $\ell=2$.
Note that $I(u)$ induces $\mathrm{Her}(d-\dim u,\ell)$.
By \cite[Theorem 21]{Hemmeter1986EJC}, the maximal cliques of $\Gamma$ are the $I(u)$ with $\dim u=d-1$.
Comparing the classical parameters it follows that

\begin{thm}\label{Hermitean forms graph}
$\mathrm{Her}(d,\ell)$ has no nontrivial descendent.
\end{thm}

\subsection*{Alternating forms graphs}

Let $\ell$ be a prime power and $\Delta$ the dual polar graph on $[D_n(\ell)]$, where $n\in\{2d,2d+1\}$.
Fix a vertex $z$ of $\Delta$.
The subgraph $\Gamma=\mathrm{Alt}(n,\ell)$ of the distance-$2$ graph of $\Delta$ induced on $X:=\Delta_n(z)$ is the alternating forms graph \cite[\S 9.5B]{BCN1989B}, i.e., $\Gamma=\Delta_n^2(z)$ with the notation of \S \ref{sec: bipartite case}.
$\Gamma$ has classical parameters $(d,\ell^2,\ell^2-1,\ell^m-1)$.

By \cite[Proposition 5.26]{Lambeck1990D}, every noncomplete convex subset of $X$ is of the form $I(u)=\{x\in X:u\subseteq x\}$ where $u$ is an isotropic subspace of $[D_n(\ell)]$ with $\dim u\leqslant n-4$ and $z\cap u=0$.
Note that $I(u)$ induces $\mathrm{Alt}(n-\dim u,\ell)$.
By \cite[Lemma 19]{Hemmeter1986EJC}, $\Gamma$ has clique number $\ell^{n-1}$ and the maximum cliques are of the form $C(x,v)=\{x\}\cup\{y\in\Gamma(x):x\cap y\subseteq v\}$ where $x\in X$ and $v$ is a subspace of $x$ with $\dim v=n-1$.
It follows that

\begin{thm}\label{alternating forms graph}
Let $Y$ be a nontrivial descendent of $\mathrm{Alt}(n,\ell)$.
Then $n=2d$ and $Y$ takes one of the following forms:
(i) $w=1$ and $Y=C(x,v)$;
(ii) $w=d-1$ and $Y=I(u)$ with $\dim u=1$.
\end{thm}

\subsection*{Quadratic forms graphs}

Let $\ell$ be a prime power and $X$ the set of quadratic forms on $V=\mathbb{F}_{\ell}^{n-1}$ over $\mathbb{F}_{\ell}$, where $n\in\{2d,2d+1\}$.
For $x\in X$ let $\mathrm{Rad}\,x=x^{-1}(0)\cap\mathrm{Rad}\,B_x$ and $\mathrm{rk}\,x=\dim(V/\mathrm{Rad}\,x)$, where $B_x$ is the symmetric bilinear form associated with $x$ and $\mathrm{Rad}\,B_x$ denotes its radical.
Let $\Gamma=\mathrm{Quad}(n-1,\ell)$ be the quadratic forms graph on $V$ \cite[\S 9.6]{BCN1989B}:
$x,y\in X$ are adjacent if $\mathrm{rk}(x-y)=1$ or $2$.
$\Gamma$ has classical parameters $(d,\ell^2,\ell^2-1,\ell^m-1)$.
We remark that if $\ell$ is odd then $\Gamma$ is isomorphic to the subgraph of the distance $1$-or-$2$ graph of the dual polar graph $\Sigma$ on $[C_{n-1}(\ell)]$ induced on $\Sigma_{n-1}(z)$ with $z \in V \Sigma$, or equivalently, $\Gamma$ is isomorphic to $\Delta_n^2(z)$, where $\Delta=\mathrm{Hem}_n(\ell)$ and $z \in V \! \Delta$; see \cite{CHM2008DM}.

By \cite[Proposition 5.36]{Lambeck1990D}, \cite[Theorem 1.2]{MPS1993JAC}, every noncomplete convex subset of $X$ is of the form\footnote{\label{embedding in symplectic dual polar graph}If $\ell$ is odd, then in terms of the above identification $X=\Sigma_{n-1}(z)$, every noncomplete convex subgraph of $\Gamma$ is rewritten as $I(u)=\{y\in X:u\subseteq y\}$ for some isotropic subspace of $[C_{n-1}(\ell)]$ with $\dim u\leqslant n-4$ and $z\cap u = 0$.} $I(x,u)=\{y\in X:\mathrm{Rad}(x-y)\supseteq u\}$ where $x\in X$ and $u$ is a subspace of $V$ with $\dim u\leqslant n-4$.
Note that $I(x,u)$ induces $\mathrm{Quad}(n-1-\dim u,\ell)$.
By \cite[Theorems 14, 24, 26]{Hemmeter1988EJC}, $\Gamma$ has clique number $\ell^{n-1}$ and there are two types of maximum cliques:
A \emph{type} $1$ clique is defined to be $C(x)=\{y\in X:\mathrm{rk}(x-y)\leqslant 1\}$ where $x\in X$.
If $q$ is odd then a \emph{type} $2$ clique is defined to be $C(x,u)=\{y\in X:B_y|_{u\times u}=B_x|_{u\times u}\}$ where $x\in X$ and $u$ is a subspace of $V$ with $\dim u=n-2$.
If $q$ is even then it is defined to be $C(x,u,\gamma,a)=\{y\in X:B_y|_{u\times u}=B_x|_{u\times u},\ (x-y)(\xi)=a((B_x-B_y)(\gamma,\xi))^2\ \text{for all}\ \xi\in u\}$ where $x\in X$, $u$ a subspace of $V$ with $\dim u=n-2$, $\gamma\in V\backslash u$ and $a\in\mathbb{F}_{\ell}$.
See also \cite{BHW1995EJC,HW1999EJC}.
Cliques of types $1$ and $2$ are called \emph{grand cliques}.
It follows that

\begin{thm}\label{quadratic forms graph}
Let $Y$ be a nontrivial descendent of $\mathrm{Quad}(n-1,\ell)$.
Then $n=2d$ and $Y$ takes one of the following forms:
(i) $w=1$ and $Y$ is a grand clique;
(ii) $w=d-1$ and $Y=I(x,u)$ with $\dim u=1$.
\end{thm}

\subsection*{Unitary dual polar graphs with second $Q$-polynomial ordering}

Let $\ell$ be a prime power and $\Gamma=U(2d,\ell)$ the dual polar graph on $[^2\! A_{2d-1}(\ell)]$.
Besides the ordinary $(d,\ell^2,0,\ell)$, $\Gamma$ has another classical parameters $(d,-\ell,\alpha,\beta)$, where $\alpha+1=\frac{1+\ell^2}{1-\ell}$ and $\beta+1=\frac{1-(-\ell)^{d+1}}{1-\ell}$ \cite[Corollary 6.2.2]{BCN1989B}.
Here we consider the standard $Q$-polynomial ordering with respect to the latter classical parameters; it is $\{E_0,E_d,E_1,E_{d-1},\dots\}$ in terms of the ordinary ordering $\{E_i\}_{i=0}^d$.

By \cite[Proposition 5.19]{Lambeck1990D}, every noncomplete convex subset of $X$ is of the form $I(u)=\{x\in X:u\subseteq x\}$ for some isotropic subspace $u$ of $[^2\! A_{2d-1}(\ell)]$ with $\dim u=w^*$.
By \cite[Lemma 10]{Hemmeter1986EJC}, the maximal cliques of $\Gamma$ are the $I(u)$ with $\dim u=d-1$.
It follows that

\begin{thm}\label{unitary dual polar graph}
$U(2d,\ell)$ (with second $Q$-polynomial ordering) has no nontrivial descendent.
\end{thm}

\subsection*{Half dual polar graphs}

Let $\ell$ be a prime power and $\Delta$ the bipartite dual polar graph on $[D_n(\ell)]$ with $V \! \Delta=X^+ \cup X^-$, where $n \in \{2d,2d+1\}$ and $X^+,X^-$ are bipartite halves of $\Delta$.
The path-length distance for $\Delta$ is denoted $\partial_{\Delta}$.
Let $\Gamma=D_{n,n}(\ell)$ be a halved graph of $\Delta$ with vertex set $X=X^{\varepsilon}$ where $\varepsilon\in\{+,-\}$, whence $2\partial=\partial_{\Delta}|_{X\times X}$ \cite[\S 9.4C]{BCN1989B}.
$\Gamma$ has classical parameters $(d,\ell^2,\alpha,\beta)$, where $\alpha+1=\gauss{3}{1}_{\ell}$ and $\beta+1=\gauss{m+1}{1}_{\ell}$.
Pepe, Storme and Vanhove \cite[\S 5]{PSV2010pre} recently classified the descendents of $\Gamma$ when $n=2d$ and $w=d-1$, partly based on the Erd\H{o}s--Ko--Rado theorem for Grassmann graphs (cf.~\cite[Theorem 3]{Tanaka2006JCTA}).
Our approach below uses \eqref{alternating forms graph} instead.

First suppose $w>1$.
Fix $x \in Y$ and pick any $z\in\Delta_{n-2w}(x)$ such that $\Delta_n(z)\cap Y\ne\emptyset$.
By \eqref{intersection of descendent of bipartite half with last subconstituent}, \eqref{alternating forms graph} we find $n=2d$, $w=d-1$ and there is an isotropic subspace $u$ of $[D_{2d}(\ell)]$ with $\dim u=1$ and $z\cap u=0$ such that $Y_0:=\Delta_{2d}(z)\cap Y=\{y\in\Delta_{2d}(z):u\subseteq y\}$.
Note that $z\in\Gamma(x)\backslash Y$ and $Y_0=\Gamma_d(z)\cap Y\subseteq (\Gamma_Y)_{d-1}(x)$.
For convenience, let $\Xi$ be the dual polar graph on the residual polar space of $u$, so that $V\Xi=\{y\in V\!\Delta:u\subseteq y\}$.

\begin{note}\label{x contains u: half dual polar}
$x\in V\Xi$.
\end{note}

\begin{proof}
Suppose $u \not\subseteq x$.
Let $x^{\dagger}=u+(x\cap u^{\perp}),z^{\dagger}=u+(z\cap u^{\perp})\in V\Xi$.
Then $\partial_{\Delta}(x,y)=\partial_{\Delta}(x^{\dagger},y)+1$ for any $y\in V\Xi$, and similarly for $z^{\dagger}$.
Since $Y_0\subseteq\Xi_{2d-3}(x^{\dagger})$ we find $x^{\dagger}\in\Delta_3(z)$, i.e., $x^{\dagger}\in\Xi_2(z^{\dagger})$.
But then $\Gamma_d(x)\cap Y_0=\Xi_{2d-1}(x^{\dagger})\cap\Xi_{2d-1}(z^{\dagger})\ne\emptyset$, a contradiction.
\end{proof}

\begin{note}\label{elements of last subconstituent contain u: half dual polar}
$(\Gamma_Y)_{d-1}(x)\subseteq V\Xi$.
\end{note}

\begin{proof}
Let $z_1\in\Gamma(x)\backslash Y$.
Since $Y$ is completely regular in $\Gamma$, $Y_1:=\Gamma_d(z_1)\cap Y\ne\emptyset$.
Let $u_1$ be the isotropic subspace of $[D_{2d}(\ell)]$ with $\dim u_1=1$, $z_1\cap u_1=0$ and $Y_1=\{y\in\Gamma_d(z_1):u_1 \subseteq y\}$.
Note that $|Y_0|=|Y_1|=\ell^{(d-1)(2d-1)}$ and $|(\Gamma_Y)_{d-1}(x)|=\prod_{i=1}^{d-1}(b_{i-1}(\Gamma_Y)/c_i(\Gamma_Y))=\ell^{(d-1)(2d-3)}\gauss{2d-1}{1}_{\ell}$.
Suppose $u\ne u_1$.
Then $u+u_1\subseteq x$ by \eqref{x contains u: half dual polar}, so that by looking at the residual polar space of $u+u_1$ we find $|Y_0\cap Y_1|\leqslant |\mathrm{Alt}(2d-2,\ell)|=\ell^{(d-1)(2d-3)}$.
If $\ell>2$ then $|Y_0\cap Y_1|\geqslant |Y_0|+|Y_1|-|(\Gamma_Y)_{d-1}(x)|>\ell^{(d-1)(2d-3)}$, a contradiction.
Hence $\ell=2$, $Y_0\cup Y_1=(\Gamma_Y)_{d-1}(x)$ and $Y_0\cap Y_1=\{y\in(\Gamma_Y)_{d-1}(x):u+u_1\subseteq y\}$, from which it follows that $Y_0=\{y\in(\Gamma_Y)_{d-1}(x):u\subseteq y\}$, $Y_1=\{y\in(\Gamma_Y)_{d-1}(x):u_1\subseteq y\}$.
Pick $v \in Y_0\cap Y_1$.
Since $a_{d-1}(\Gamma)>a_{d-1}(\Gamma_Y)$ there is $z_2 \in \Gamma(x)\cap\Gamma_{d-1}(v)\backslash Y$.
Let $u_2$ be the isotropic subspace of $[D_{2d}(2)]$ with $\dim u_2=1$, $z_2\cap u_2=0$ and $Y_2:=\Gamma_d(z_2)\cap Y=\{y\in\Gamma_d(z_2):u_2\subseteq y\}$.
Since $v \not \in Y_2$ we find $Y_2 \ne Y_0,Y_1$, so that $u_2\ne u,u_1$, and thus $|Y_0\cap Y_2|,|Y_1\cap Y_2|\leqslant 2^{(d-1)(2d-3)}$ by the above argument.
But $\max \{|Y_0\cap Y_2|,|Y_1 \cap Y_2|\} \geqslant \frac{1}{2}|Y_2|>2^{(d-1)(2d-3)}$, a contradiction.
Hence $u=u_1$.
Since $(\Gamma_Y)_{d-1}(x)=\bigcup_{z_1\in\Gamma(x)\backslash Y}\Gamma_d(z_1)\cap Y$, the proof is complete.
\end{proof}

It follows from \eqref{x contains u: half dual polar}, \eqref{elements of last subconstituent contain u: half dual polar} that $Y\subseteq V\Xi\cap X$, so that $\Gamma_Y$ is a subgraph of a halved graph of $\Xi$.
Since $\iota(\Gamma_Y)=\iota(D_{2d-1,2d-1}(\ell))$ we conclude $Y=V\Xi\cap X$.

Next suppose $w=1$.
Then $|Y|=\gauss{m+1}{1}_{\ell}$.
By \cite[Lemma 12]{Hemmeter1986EJC} or \cite[Theorem 3.5]{BH1992EJC}, $\Gamma$ has clique number $\gauss{2d}{1}_{\ell}$ and the maximum cliques are of the form $\Delta(z)$ where $z\in V\!\Delta\backslash X$.

\begin{thm}\label{half dual polar graph}
Let $Y$ be a nontrivial descendent of $D_{n,n}(\ell)$.
Then $n=2d$ and one of the following holds:
(i) $w=1$ and $Y=\Delta(z)$ for some $z\in V \!\Delta\backslash X$;
(ii) $w=d-1$ and $Y=\{x \in X:u \subseteq x\}$ for some isotropic subspace $u$ of $[D_{2d}(\ell)]$ with $\dim u=1$.
\end{thm}

\subsection*{Ustimenko graphs}

Let $\ell$ be an \emph{odd} prime power and $\Sigma$ the dual polar graph on $[C_{n-1}(\ell)]$ with vertex set $X=V\Sigma$, where $n\in\{2d,2d+1\}$.
The path-length distance for $\Sigma$ is denoted $\partial_{\Sigma}$.
The Ustimenko graph $\Gamma=\mathrm{Ust}_{n-1}(\ell)$ is the distance $1$-or-$2$ graph of $\Sigma$, whence $\partial=\lceil\frac{1}{2}\partial_{\Sigma}\rceil$ \cite[\S 9.4C]{BCN1989B}.
$\Gamma$ has classical parameters $(d,\ell^2,\alpha,\beta)$, where $\alpha+1=\gauss{3}{1}_{\ell}$ and $\beta+1=\gauss{m+1}{1}_{\ell}$.
Note that $\Gamma$ may also be viewed as a halved graph of $\Delta=\mathrm{Hem}_n(\ell)$.
Pepe et al.~\cite[\S 7]{PSV2010pre} classified the descendents of $\Gamma$ when $n=2d$ and $w=d-1$ by a different approach from the following.

First suppose $w>1$.
Fix $x\in Y$ and pick any $z\in\Delta_{n-2w}(x)$ such that $\Delta_n(z)\cap Y\ne\emptyset$.
By \eqref{intersection of descendent of bipartite half with last subconstituent}, \eqref{quadratic forms graph} we find $n=2d$, $w=d-1$ and there is an isotropic subspace $u$ of $[C_{2d-1}(\ell)]$ with $\dim u=1$ and $z\cap u=0$ such that $Y_0:=\Delta_{2d}(z)\cap Y=\{y\in\Delta_{2d}(z):u\subseteq y\}$.\footnote{See footnote \ref{embedding in symplectic dual polar graph}.}
Note that $z\in\Gamma (x)\backslash Y$ and $Y_0=\Gamma_d(z)\cap Y=\Sigma_{2d-1}(z)\cap Y\subseteq(\Gamma_Y)_{d-1}(x)$.
Let $\Xi$ be the dual polar graph on the residual polar space of $u$, so that $V\Xi=\{y\in V\Sigma:u\subseteq y\}$.

\begin{note}\label{x contains u: Ustimenko}
$x\in V\Xi$.
\end{note}

\begin{proof}
Suppose $u\not\subseteq x$.
Let $x^{\dagger}=u+(x\cap u^{\perp}),z^{\dagger}=u+(z\cap u^{\perp})\in V\Xi$.
Then $\Gamma_d(x)\cap Y_0=\Xi_{2d-2}(x^{\dagger})\cap\Xi_{2d-2}(z^{\dagger})\ne\emptyset$, a contradiction.
\end{proof}

\begin{note}\label{elements of last subconstituent contain u: Ustimenko}
$(\Gamma_Y)_{d-1}(x)\subseteq V\Xi$.
\end{note}

\begin{proof}
Let $z_1\in\Gamma(x)\backslash Y$.
Since $Y$ is completely regular in $\Gamma$, $Y_1:=\Gamma_d(z_1)\cap Y\ne\emptyset$.
Let $u_1$ be the isotropic subspace of $[C_{2d-1}(\ell)]$ with $\dim u_1=1$, $z_1\cap u_1=0$ and $Y_1=\{y\in\Gamma_d(z_1):u_1\subseteq y\}$.
Note that $|Y_0|=|Y_1|=\ell^{(d-1)(2d-1)}$ and $|(\Gamma_Y)_{d-1}(x)|=\ell^{(d-1)(2d-3)}\gauss{2d-1}{1}_{\ell}$.
Suppose $u\ne u_1$.
Then $u+u_1\subseteq x$ by \eqref{x contains u: Ustimenko}, so that $|Y_0\cap Y_1|\leqslant |\mathrm{Quad}(2d-3,\ell)|=\ell^{(d-1)(2d-3)}$.
But $|Y_0\cap Y_1|\geqslant |Y_0|+|Y_1|-|(\Gamma_Y)_{d-1}(x)|>\ell^{(d-1)(2d-3)}$, a contradiction.
Hence $u=u_1$.
Since $(\Gamma_Y)_{d-1}(x)=\bigcup_{z_1\in\Gamma(x)\backslash Y}\Gamma_d(z_1)\cap Y$, the proof is complete.
\end{proof}

\begin{note}\label{elements of Y contain u: Ustimenko}
$Y\subseteq V\Xi$.
\end{note}

\begin{proof}
Suppose there is $y\in Y\backslash V\Xi$.
Pick $v \in (\Gamma_Y)_{d-1}(x)$ such that $\partial(x,y)+\partial(y,v)=d-1$.
Set $y^{\dagger}=u+(y\cap u^{\perp})\in V\Xi$.
Then by \eqref{x contains u: Ustimenko}, \eqref{elements of last subconstituent contain u: Ustimenko} it follows that $\partial_{\Sigma}(x,v)\leqslant\partial_{\Sigma}(x,y^{\dagger})+\partial_{\Sigma}(y^{\dagger},v)=\partial_{\Sigma}(x,y)+\partial_{\Sigma}(y,v)-2\leqslant 2d-4$, a contradiction.
\end{proof}

By \eqref{elements of Y contain u: Ustimenko} and since $\iota(\Gamma_Y)=\iota(\mathrm{Ust}_{2d-2}(\ell))$ we conclude $Y=V\Xi$.

Next suppose $w=1$.
Then $|Y|=\gauss{m+1}{1}_{\ell}$.
By \cite[Theorem 3.7]{BH1992EJC}, $\Gamma$ has clique number $\gauss{2d}{1}_{\ell}$ and the maximum cliques are of the form $\Sigma(x) \cup \{x\}$ where $x \in X$.

\begin{thm}\label{Ustimenko graph}
Let $Y$ be a nontrivial descendent of $\mathrm{Ust}_{n-1}(\ell)$.
Then $n=2d$ and one of the following holds:
(i) $w=1$ and $Y=\Sigma(x)\cup\{x\}$ for some $x \in X$;
(ii) $w=d-1$ and $Y=\{x \in X:u\subseteq x\}$ for some isotropic subspace $u$ of $[C_{2d-1}(\ell)]$ with $\dim u=1$.
\end{thm}

\subsection*{Twisted Grassmann graphs}

Let $q$ be a prime power and fix a hyperplane $H$ of $\mathbb{F}_q^{2d+1}$.
Let $X_1$ be the set of $(d+1)$-dimensional subspaces of $\mathbb{F}_q^{2d+1}$ not contained in $H$, and $X_2$ the set of $(d-1)$-dimensional subspaces of $H$.
The twisted Grassmann graph $\Gamma=\tilde{J}_q(2d+1,d)$ \cite{DK2005IM,FKT2006IIG,BFK2009EJC,MT2009pre} has vertex set $X=X_1\cup X_2$, and $x,y\in X$ are adjacent if $\dim x+\dim y-2\dim x\cap y=2$.
$\Gamma$ has classical parameters $(d,q,q,\beta)$, where $\beta+1=\gauss{d+2}{1}_q$.
Note that $X_2$ is a descendent of $\Gamma$ and induces $J_q(2d,d-1)$.
We observe
\begin{equation} \label{distance in terms of dimension}
	2\partial(x,y)=\dim x+\dim y-2\dim x\cap y \quad (x,y\in X).
\end{equation}

\begin{note}\label{geodesic}
For $x,y,z\in X$, we have $\partial(x,z)+\partial(z,y)=\partial(x,y)$ if and only if $x\cap y\subseteq z=(x\cap z)+(y\cap z)$.
\end{note}

\begin{proof}
Observe $\dim x\cap z+\dim y\cap z\leqslant\dim z+\dim x\cap y\cap z\leqslant\dim z+\dim x\cap y$, and equality holds if and only if $x\cap y\subseteq z=(x\cap z)+(y\cap z)$.
Hence the result follows from \eqref{distance in terms of dimension}.
\end{proof}

\begin{note}\label{X_1}
Let $Z$ be a nonempty subset of $X_1$ such that $\{z\in X_1:x\cap y\subseteq z=(x\cap z)+(y\cap z)\}\subseteq Z$ for all $x,y\in Z$.
Then at least one of the following holds:
(i)~there is a subspace $u$ of $\mathbb{F}_q^{2d+1}$ with $\dim u=d-w(Z)+1$ such that $u\subseteq z$ for all $z\in Z$;
(ii)~there is a subspace $v$ of $\mathbb{F}_q^{2d+1}$ with $\dim v=d+w(Z)+1$ and not contained in $H$ such that $z\subseteq v$ for all $z\in Z$.
\end{note}

\begin{proof}
Fix $x,y\in Z$ with $\partial(x,y)=w(Z)$ and recall $\dim x\cap y=d-w(Z)+1$ by \eqref{distance in terms of dimension}.
Let $z\in X_1$.
We claim that if $x\cap y\not\subseteq z\not\subseteq x+y$ then $z\not\in Z$.
Suppose $z\in Z$.
Let $\gamma\in (x\cap y)\backslash z$ and $\sigma\in z\backslash ((x+y)\cup H)$.
Let $E$ be a complementary subspace of $\langle\gamma\rangle$ in $x$ such that $x\cap z\subseteq E$.
Set $z^{\dagger}=E+\langle\sigma\rangle\in X_1$.
Since $x\cap z\subseteq z^{\dagger}=(x\cap z^{\dagger})+(z\cap z^{\dagger})$ we find $z^{\dagger}\in Z$.
But then $\gamma\not\in z^{\dagger}\cap y (=E\cap y)\subseteq x\cap y$ implies $\dim z^{\dagger}\cap y<d-w(Z)+1$ and thus $\partial(z^{\dagger},y)>w(Z)$ by \eqref{distance in terms of dimension}, a contradiction.
Hence the claim follows.
It follows that every $z\in Z$ satisfies $x\cap y\subseteq z$ or $z\subseteq x+y$ (or both).
Next we claim that there is no pair $(z_1,z_2)$ of elements of $Z$ such that $x\cap y\not\subseteq z_1\subseteq x+y$ and $x\cap y\subseteq z_2\not\subseteq x+y$.
Suppose such a pair $(z_1,z_2)$ exists.
Let $\zeta\in z_2\backslash ((x+y)\cup H)$.
Let $E$ be a hyperplane of $z_1$ such that $z_1\cap z_2\subseteq E$.
Set $z^{\ddagger}=E+\langle\zeta\rangle\in X_1$.
Since $z_1\cap z_2\subseteq z^{\ddagger}=(z_1\cap z^{\ddagger})+(z_2\cap z^{\ddagger})$ we find $z^{\ddagger}\in Z$.
But this is absurd since $x\cap y\not\subseteq z^{\ddagger}\not\subseteq x+y$.
Hence the claim follows.
Setting $u=x\cap y$ and $v=x+y$, we find that at least one of (i), (ii) holds.
\end{proof}

\begin{note}\label{X_1, X_2}
Let $Z$ be a nonempty convex subset of $X$.
If there are vertices $x\in Z\cap X_1$, $y\in Z\cap X_2$ with $\partial(x,y)=w(Z)$, then for each $i=1,2$ at least one of the following holds:
(i) there is a subspace $u$ of $H$ with $\dim u=d-w(Z)$ such that $u\subseteq z$ for all $z\in Z\cap X_i$;
(ii) there is a subspace $v$ of $\mathbb{F}_q^{2d+1}$ with $\dim v=d+w(Z)$ and not contained in $H$ such that $z \subseteq v$ for all $z\in Z\cap X_i$.
\end{note}

\begin{proof}
Similar to the proof of \eqref{X_1}, by virtue of \eqref{geodesic}.
\end{proof}

\begin{thm}\label{twisted Grassmann graph}
Let $Y$ be a nontrivial descendent of $\tilde{J}_q(2d+1,d)$.
Then $Y=\{x \in X_2:u \subseteq x\}$ for some subspace $u$ of $H$ with $\dim u=d-w-1$.
\end{thm}

\begin{proof}
Note that $\iota(\Gamma_Y)=\iota(J_q(d+w+1,w))$, whence $|Y|=\gauss{d+w+1}{w}_q$.
First suppose $Y\subseteq X_1$.
Then in view of \eqref{geodesic}, if \eqref{X_1}(i) holds then $|Y|\leqslant\gauss{d+w}{w}_q<\gauss{d+w+1}{w}_q$, and if \eqref{X_1}(ii) holds then $|Y|\leqslant\gauss{d+w+1}{d+1}_q-\gauss{d+w}{d+1}_q<\gauss{d+w+1}{w}_q$, a contradiction.
Hence $Y\not\subseteq X_1$.
Next suppose $Y\cap X_1$, $Y\cap X_2$ are nonempty.
Let $z\in Y\cap X_1$, $y\in Y\cap X_2$, and pick any $x\in Y$ with $\partial(x,y)=w$, $\partial(x,z)=w-\partial(z,y)$.
Since $X_2$ is convex we find $x\in Y\cap X_1$.
Let $w_1:=w(Y\cap X_1)\leqslant w$.
By \eqref{geodesic}, $Y\cap X_1$ satisfies the assumptions of \eqref{X_1}.
If $Y\cap X_1$ satisfies \eqref{X_1}(i) then $|Y\cap X_1|\leqslant\gauss{d+w_1}{w_1}_q\leqslant\gauss{d+w}{w}_q$.
If $Y\cap X_1$ satisfies \eqref{X_1, X_2}(ii) then $|Y\cap X_1|\leqslant\gauss{d+w}{d+1}_q\leqslant\gauss{d+w}{w}_q$.
If $Y\cap X_1$ satisfies both \eqref{X_1}(ii) and \eqref{X_1, X_2}(i) then $|Y\cap X_1|\leqslant\gauss{w+w_1+1}{w+1}_q\leqslant\gauss{d+w}{w}_q$ since $w<d$.
Hence we always have $|Y\cap X_1|\leqslant\gauss{d+w}{w}_q$.
On the other hand, if $Y\cap X_2$ satisfies \eqref{X_1, X_2}(i) then $|Y\cap X_2|\leqslant\gauss{d+w}{w-1}_q$, and if $Y\cap X_2$ satisfies \eqref{X_1, X_2}(ii) then $|Y\cap X_2|\leqslant\gauss{d+w-1}{w}_q$.
Since $q^{d+1} \gauss{d+w}{w-1}_q> \gauss{d+w-1}{w}_q$, it follows that $|Y|=|Y \cap X_1|+|Y \cap X_2|< \gauss{d+w}{w}_q+q^{d+1} \gauss{d+w}{w-1}_q= \gauss{d+w+1}{w}_q$, a contradiction.
Hence $Y \subseteq X_2$, and by \eqref{transitivity}, \eqref{truncated projective geometry}, $Y$ must be of the form as in \eqref{twisted Grassmann graph}.
\end{proof}

\subsection*{Summary and remarks}

Let $\mathscr{P}$ be the set of descendents of $\Gamma$ and $w(\mathscr{P})=\{w(Y):Y\in\mathscr{P}\}$.
In the table below, we list $w(\mathscr{P})$ for each of the $15$ families of graphs with classical parameters.

\begin{center}
\begin{tabular}{llc}
\hline \\[-.17in]
\multicolumn{1}{c}{$\Gamma$} & \multicolumn{1}{c}{\#} & $w(\mathscr{P})\backslash\{0,d\}$ \\[.03in]
\hline \\[-.17in]
$J(\nu,d)$ ($\nu \geqslant 2 d$) & \eqref{truncated Boolean algebra} & $\{1,2,\dots,d-1\}$ \\[.03in]
$H(d,\ell)$ ($\ell \geqslant 2$) & \eqref{Hamming matroid} & $\{1,2,\dots,d-1\}$ \\[.03in]
$J_q(\nu,d)$ ($\nu \geqslant 2 d$) & \eqref{truncated projective geometry} & $\{1,2,\dots,d-1\}$ \\[.03in]
$\mathrm{Bil}_q(d,e)$ ($e \geqslant d$) & \eqref{attenuated space} & $\{1,2,\dots,d-1\}$ \\[.03in]
Dual polar graph & \eqref{polar spaces} & $\{1,2,\dots,d-1\}$ \\[.03in]
\hline \\[-.17in]
$\mathrm{Doob}(d_1,d_2)$ ($d=2d_1+d_2$) & \eqref{Doob graph} & $\{1,2,\dots,d-1\}$ \\[.03in]
$\mathrm{Hem}_d(q)$ ($q$ : odd) & \eqref{Hemmeter graph} & $\{1,2,\dots,d-1\}$ \\[.03in]
$\tilde{J}_q(2d+1,d)$ & \eqref{twisted Grassmann graph} & $\{1,2,\dots,d-1\}$ \\[.03in]
\hline \\[-.17in]
$\frac{1}{2}H(n,2)$ & \eqref{halved cube} & $\{1,d-1\}$ ($n=2d$), $\emptyset$ ($n=2d+1$)\\[.03in]
$\mathrm{Her}(d,\ell)$ & \eqref{Hermitean forms graph} & $\emptyset$ \\[.03in]
$\mathrm{Alt}(n,\ell)$ & \eqref{alternating forms graph} & $\{1,d-1\}$ ($n=2d$), $\emptyset$ ($n=2d+1$) \\[.03in]
$\mathrm{Quad}(n-1,\ell)$ & \eqref{quadratic forms graph} & $\{1,d-1\}$ ($n=2d$), $\emptyset$ ($n=2d+1$) \\[.03in]
$U(2d,\ell)$ & \eqref{unitary dual polar graph} & $\emptyset$ \\[.03in]
$D_{n,n}(\ell)$ & \eqref{half dual polar graph} & $\{1,d-1\}$ ($n=2d$), $\emptyset$ ($n=2d+1$) \\[.03in]
$\mathrm{Ust}_{n-1}(\ell)$ ($\ell$ : odd) & \eqref{Ustimenko graph} & $\{1,d-1\}$ ($n=2d$), $\emptyset$ ($n=2d+1$) \\[.03in]
\hline
\end{tabular}
\end{center}

Note that the $5$ families of the first group in the table are associated with regular semilattices, and that the $3$ families of the second group have the classical parameters of graphs belonging to the first group.
It should be remarked however that, for every $i$ $(0<i<d)$, the graphs in the second group possess pairs of vertices at distance $i$ which are not contained in any descendent with width $i$, with the exception of $(\Gamma,i)=(\mathrm{Hem}_d(q),1)$.

In \eqref{from Gamma_Y to Gamma} we posed the problem of determining the filter $V_{[\Gamma]}$ of the poset $\mathscr{L}$ generated by the isomorphism class $[\Gamma]$.
We end this section with describing $V_{[\Gamma]}$ for some examples where $\Gamma$ has classical parameters $(d,q,\alpha,\beta)$ with $q=1$ or $q<-1$.
The distance-regular graphs with classical parameters $(d,1,\alpha,\beta)$ $(d\geqslant 3)$ are known: the Johnson graphs, Hamming graphs, Doob graphs, halved cubes and the Gosset graph; see e.g., \cite[Theorem 6.1.1]{BCN1989B}.
By virtue of \eqref{descendents inherit classical parameters}, \eqref{truncated Boolean algebra}, \eqref{Hamming matroid}, \eqref{Doob graph} and \eqref{halved cube}, we have the following table:
\begin{center}
\begin{tabular}{ll}
\hline
\\[-.17in]
\multicolumn{1}{c}{$\Gamma$} & \multicolumn{1}{c}{$V_{[\Gamma]}$} \\[.02in]
\hline
\\[-.15in]
$J(\nu,d)$ $(\nu\geqslant 2d)$ & $\{[J(\nu+i,d+i)]:0\leqslant i\leqslant\nu-2d\}$ \\[.06in]
$H(d,\ell)$ $(\ell\ne 4)$ & $\{[H(e,\ell)]:e\geqslant d\}$ \\[.06in]
$H(d,4)$ & $\{[H(e,4)]:e\geqslant d\}\cup\{[\mathrm{Doob}(d_1,d_2)]:d_2\geqslant d\}$ \\[.06in]
$\mathrm{Doob}(d_1,d_2)$ & $\{[\mathrm{Doob}(e_1,e_2)]:e_1\geqslant d_1,\ e_2\geqslant d_1\}$ \\[.06in]
$\frac{1}{2}H(2d,2)$ & $\{[\frac{1}{2}H(2d,2)]\}$ \\[.06in]
$\frac{1}{2}H(2d+1,2)$ & $\{[\frac{1}{2}H(2d+1,2)],[\frac{1}{2}H(2d+2,2)]\}$ \\[.06in]
\hline
\end{tabular}
\end{center}

Weng \cite{Weng1999JCTB} showed that if a distance-regular graph $\Delta$ having classical parameters $(d,q,\alpha,\beta)$ satisfies $d\geqslant 4$, $q<-1$, $a_1\ne 0$, $c_2>1$ then either (i) $\Delta=\mathrm{Her}(d,\ell)$ ($q=-\ell$); (ii)~$\Delta=U(2d,\ell)$ ($q=-\ell$); or (iii) $\alpha=(q-1)/2$, $\beta=-(1+q^d)/2$ and $-q$ is a power of an odd prime.
Hence, by \eqref{descendents inherit classical parameters}, \eqref{Hermitean forms graph}, \eqref{unitary dual polar graph} it follows that $[\mathrm{Her}(d,\ell)]$, $[U(2d,\ell)]$ are maximal elements in $\mathscr{L}$ for all $d\geqslant 3$, so that the filter generated by any of them is a singleton.
It would be interesting if the poset $\mathscr{L}$ is of some use, e.g., in the classification of distance-regular graphs $\Gamma$ with $\iota(\Gamma)=\iota(J_q(2d+1,d))=\iota(\tilde{J}_q(2d+1,d))$.

\appendix

\section{The list of parameter arrays}

We display the parameter arrays of Leonard systems.
The data in \eqref{list of parameter arrays} is taken from \cite{Terwilliger2005DCC}, but the presentation is changed so as to be consistent with the notation in \cite{BI1984B,Terwilliger1992JAC,Terwilliger1993JACa,Terwilliger1993JACb}.

\begin{thm}[{\cite[Theorem 5.16]{Terwilliger2005DCC}}]\label{list of parameter arrays}
Let $\Phi$ be the Leonard system from \eqref{Leonard system} and let $p(\Phi)=\left(\{\theta_i\}_{i=0}^d;\{\theta_i^*\}_{i=0}^d;\{\varphi_i\}_{i=1}^d;\{\phi_i\}_{i=1}^d\right)$ be as in \eqref{parameter array}.
Then at least one of the following cases I, IA, II, IIA, IIB, IIC, III hold:
\begin{itemize}
\item[(I)] $p(\Phi)=p(\mathrm{I};q,h,h^*,r_1,r_2,s,s^*,\theta_0,\theta_0^*,d)$ where $r_1r_2=ss^*q^{d+1}$,
\begin{align*}
	\theta_i&=\theta_0+h(1-q^i)(1-sq^{i+1})q^{-i}, \\
	\theta_i^*&=\theta_0^*+h^*(1-q^i)(1-s^*q^{i+1})q^{-i}
\end{align*}
for $0\leqslant i\leqslant d$, and
\begin{align*}
	\varphi_i&=hh^*q^{1-2i}(1-q^i)(1-q^{i-d-1})(1-r_1q^i)(1-r_2q^i), \\
	\phi_i&=\begin{cases} hh^*q^{1-2i}(1-q^i)(1-q^{i-d-1})(r_1-s^*q^i)(r_2-s^*q^i)/s^* & \text{if} \ s^*\ne 0, \\ hh^*q^{d+2-2i}(1-q^i)(1-q^{i-d-1})(s-r_1q^{i-d-1}-r_2q^{i-d-1}) & \text{if} \ s^*=0 \end{cases}
\end{align*}
for $1\leqslant i\leqslant d$.
\item[(IA)] $p(\Phi)=p(\mathrm{IA};q,h^*,r,s,\theta_0,\theta_0^*,d)$ where
\begin{align*}
	\theta_i&=\theta_0-sq(1-q^i), \\
	\theta_i^*&=\theta_0^*+h^*(1-q^i)q^{-i}
\end{align*}
for $0\leqslant i\leqslant d$, and
\begin{align*}
	\varphi_i&=-rh^*q^{1-i}(1-q^i)(1-q^{i-d-1}), \\
	\phi_i&=h^*q^{d+2-2i}(1-q^i)(1-q^{i-d-1})(s-rq^{i-d-1})
\end{align*}
for $1\leqslant i\leqslant d$.
\item[(II)] $p(\Phi)=p(\mathrm{II};h,h^*,r_1,r_2,s,s^*,\theta_0,\theta_0^*,d)$ where $r_1+r_2=s+s^*+d+1$,
\begin{align*}
	\theta_i&=\theta_0+hi(i+1+s), \\
	\theta_i^*&=\theta_0^*+h^*i(i+1+s^*)
\end{align*}
for $0\leqslant i\leqslant d$, and
\begin{align*}
	\varphi_i&=hh^*i(i-d-1)(i+r_1)(i+r_2), \\
	\phi_i&=hh^*i(i-d-1)(i+s^*-r_1)(i+s^*-r_2)
\end{align*}
for $1\leqslant i\leqslant d$.
\item[(IIA)] $p(\Phi)=p(\mathrm{IIA};h,r,s,s^*,\theta_0,\theta_0^*,d)$ where
\begin{align*}
	\theta_i&=\theta_0+hi(i+1+s), \\
	\theta_i^*&=\theta_0^*+s^*i
\end{align*}
for $0\leqslant i\leqslant d$, and
\begin{align*}
	\varphi_i&=hs^*i(i-d-1)(i+r), \\
	\phi_i&=hs^*i(i-d-1)(i+r-s-d-1)
\end{align*}
for $1\leqslant i\leqslant d$.
\item[(IIB)] $p(\Phi)=p(\mathrm{IIB};h^*,r,s,s^*,\theta_0,\theta_0^*,d)$ where
\begin{align*}
	\theta_i&=\theta_0+si, \\
	\theta_i^*&=\theta_0^*+h^*i(i+1+s^*)
\end{align*}
for $0\leqslant i\leqslant d$, and
\begin{align*}
	\varphi_i&=h^*si(i-d-1)(i+r), \\
	\phi_i&=-h^*si(i-d-1)(i+s^*-r)
\end{align*}
for $1\leqslant i\leqslant d$.
\item[(IIC)] $p(\Phi)=p(\mathrm{IIC};r,s,s^*,\theta_0,\theta_0^*,d)$ where
\begin{align*}
	\theta_i&=\theta_0+si, \\
	\theta_i^*&=\theta_0^*+s^*i
\end{align*}
for $0\leqslant i\leqslant d$, and
\begin{align*}
	\varphi_i&=ri(i-d-1), \\
	\phi_i&=(r-ss^*)i(i-d-1)
\end{align*}
for $1\leqslant i\leqslant d$.
\item[(III)] $p(\Phi)=p(\mathrm{III};h,h^*,r_1,r_2,s,s^*,\theta_0,\theta_0^*,d)$ where $r_1+r_2=-s-s^*+d+1$,
\begin{align*}
	\theta_i&=\theta_0+h(s-1+(1-s+2i)(-1)^i), \\
	\theta_i^*&=\theta_0^*+h^*(s^*-1+(1-s^*+2i)(-1)^i)
\end{align*}
for $0\leqslant i\leqslant d$, and
\begin{align*}
	\varphi_i&=\begin{cases} -4hh^*i(i+r_1) & \text{if $i$ even, $d$ even}, \\ -4hh^*(i-d-1)(i+r_2) & \text{if $i$ odd, $d$ even}, \\ -4hh^*i(i-d-1) & \text{if $i$ even, $d$ odd}, \\ -4hh^*(i+r_1)(i+r_2) & \text{if $i$ odd, $d$ odd}, \end{cases} \\
	\phi_i&=\begin{cases} 4hh^*i(i-s^*-r_1) & \text{if $i$ even, $d$ even}, \\ 4hh^*(i-d-1)(i-s^*-r_2) & \text{if $i$ odd, $d$ even}, \\ -4hh^*i(i-d-1) & \text{if $i$ even, $d$ odd}, \\ -4hh^*(i-s^*-r_1)(i-s^*-r_2) & \text{if $i$ odd, $d$ odd} \end{cases}
\end{align*}
for $1\leqslant i\leqslant d$.
\end{itemize}
\end{thm}

We call an array $p(\cdot;\cdots)$ in \eqref{list of parameter arrays} \emph{feasible} if it is the parameter array of an actual Leonard system.
See \cite[Examples 5.3--5.15]{Terwilliger2005DCC} for the explicit feasibility conditions of these arrays.
The following is essentially\footnote{The scalars $\varphi_i$, $\phi_i$ first appeared in \cite{Terwilliger2001LAA}. Hence we reproduce \cite[Lemma 2.4]{Terwilliger1992JAC} in terms of $p(\Phi)$ for accuracy purposes.} given (without proof) in \cite[Lemma 2.4]{Terwilliger1992JAC} and can be read off the proof of \cite[Theorem 5.16]{Terwilliger2005DCC}:

\begin{prop}\label{uniqueness of expression}
Referring to \eqref{list of parameter arrays}, (i)--(iv) hold below:
\begin{enumerate}
\item The following arrays are equal:
\begin{gather*}
	p(\mathrm{I};q,h,h^*,r_1,r_2,s,s^*,\theta_0,\theta_0^*,d), \\
	p(\mathrm{I};q,h,h^*,r_2,r_1,s,s^*,\theta_0,\theta_0^*,d).
\end{gather*}
\item If $ss^*\ne 0$ then the following arrays are equal:
\begin{gather*}
	p(\mathrm{I};q,h,h^*,r_1,r_2,s,s^*,\theta_0,\theta_0^*,d), \\
	p(\mathrm{I};q^{-1},hsq,h^*s^*q,r_1^{-1},r_2^{-1},s^{-1},s^{*-1},\theta_0,\theta_0^*,d).
\end{gather*}
\item The following arrays are equal:
\begin{gather*}
	p(\mathrm{II};h,h^*,r_1,r_2,s,s^*,\theta_0,\theta_0^*,d), \\
	p(\mathrm{II};h,h^*,r_2,r_1,s,s^*,\theta_0,\theta_0^*,d).
\end{gather*}
\item If $d$ is odd then the following arrays are equal:
\begin{gather*}
	p(\mathrm{III};h,h^*,r_1,r_2,s,s^*,\theta_0,\theta_0^*,d), \\
	p(\mathrm{III};h,h^*,r_2,r_1,s,s^*,\theta_0,\theta_0^*,d).
\end{gather*}
\end{enumerate}
Moreover, there is no equality among the feasible arrays $p(\cdot;\cdots)$ with $d\geqslant 3$, other than (i)--(iv) above (and combinations of (i), (ii)).
\end{prop}

Finally, we describe the parameter arrays of the $\rho$-descendents of $\Phi$:

\begin{thm}[{\cite[Theorem 6.9]{Tanaka2009LAAb}}]\label{characterization of bilinear form in parametric form}
Let $\Phi$ be the Leonard system \eqref{Leonard system} and let $p(\Phi)$ be given as in \eqref{list of parameter arrays}.
Let $\Phi'$ be a Leonard system with diameter $d'\leqslant d$.
Given an integer $\rho$ $(0\leqslant\rho\leqslant d-d')$, $\Phi'$ is a  $\rho$-descendent of $\Phi$ precisely when $p(\Phi')$ takes the following form:

\medskip\noindent
Case I:
\begin{equation*}
	p(\Phi')=p(\mathrm{I};q,h',h^{*\prime},r_1q^{\rho},r_2q^{\rho},sq^{d-d'},s^*q^{2\rho},\theta_0',\theta_0^{*\prime},d').
\end{equation*}
Case IA:
\begin{equation*}
	p(\Phi')=p(\mathrm{IA};q,h^{*\prime},r',s',\theta_0',\theta_0^{*\prime},d') \quad \text{where} \quad s'/r'=q^{d-d'-\rho}s/r.
\end{equation*}
Case II:
\begin{equation*}
	p(\Phi')=p(\mathrm{II};h',h^{*\prime},r_1+\rho,r_2+\rho,s+d-d',s^*+2\rho,\theta_0',\theta_0^{*\prime},d').
\end{equation*}
Case IIA:
\begin{equation*}
	p(\Phi')=p(\mathrm{IIA};h',r+\rho,s+d-d',s^{*\prime},\theta_0',\theta_0^{*\prime},d').
\end{equation*}
Case IIB:
\begin{equation*}
	p(\Phi')=p(\mathrm{IIB};h^{*\prime},r+\rho,s',s^*+2\rho,\theta_0',\theta_0^{*\prime},d').
\end{equation*}
Case IIC:
\begin{equation*}
	p(\Phi')=p(\mathrm{IIC};r',s',s^{*\prime},\theta_0',\theta_0^{*\prime},d') \quad \text{where} \quad s's^{*\prime}/r'=ss^*/r.
\end{equation*}
Case III, $d$ even, $d'$ even, $\rho$ even; or Case III, $d$ odd, $d'$ odd, $\rho$ even:
\begin{equation*}
	p(\Phi')=p(\mathrm{III};h',h^{*\prime},r_1+\rho,r_2+\rho,s-d+d',s^*-2\rho,\theta_0',\theta_0^{*\prime},d').
\end{equation*}
Case III, $d$ even, $d'$ even, $\rho$ odd:
\begin{equation*}
	p(\Phi')=p(\mathrm{III};h',h^{*\prime},r_2+\rho,r_1+\rho,s-d+d',s^*-2\rho,\theta_0',\theta_0^{*\prime},d').
\end{equation*}
Case III, $d$ even, $d'=1$:
\begin{equation*}
	p(\Phi')=p(\mathrm{IIC};r',s',s^{*\prime},\theta_0',\theta_0^{*\prime},1) \quad \text{where} \quad s's^{*\prime}/r'=1-\phi_{\rho+1}/\varphi_{\rho+1}.
\end{equation*}
Case III, $d$ even, $d'$ odd $\geqslant 3$; or Case III, $d$ odd, either $d'$ even or $\rho$ odd: 
Does not occur.
\end{thm}

\bigskip

\begin{ack}
The author would like to thank Tatsuro Ito, Jack Koolen, Bill Martin and Paul Terwilliger for valuable discussions and comments, and to Fr\'{e}d\'{e}ric Vanhove for sending him a preprint.
HT also wishes to thank the Department of Mathematics at the University of Wisconsin--Madison for its hospitality throughout the period in which this paper was written.
Support from the JSPS Excellent Young Researchers Overseas Visit Program is also gratefully acknowledged.
\end{ack}

\end{document}